\documentclass{amsart}
\usepackage{amsmath,amsthm,amssymb}
\usepackage{lipsum} 
\usepackage[utf8]{inputenc}
\usepackage{esint}
\usepackage[all]{xy}
\usepackage{amsfonts} 
\usepackage{xcolor}
\usepackage{tikz}
\usepackage{color}
\usepackage[colorlinks=true]{hyperref}
\usepackage{mathrsfs}
\hypersetup{
    linkcolor=cyan,          
    citecolor=orange,        
    filecolor=cyan,      
    urlcolor=cyan
}
\usepackage{bbold}
\usepackage{enumerate}
\usepackage{calc}
\usepackage{bbold}
\usepackage[colorinlistoftodos,prependcaption,textsize=tiny]{todonotes}
\usepackage{todonotes}
\usepackage[shortcuts]{extdash}
\usepackage[style=alphabetic,maxnames=200,maxalphanames=6,giveninits]{biblatex}
\makeatletter
\numberwithin{equation}{section}
\theoremstyle{plain}
        \newtheorem{theorem}{Theorem}[section]

        \newtheorem{assumption}[theorem]{Assumption}
        \newtheorem{proposition}[theorem]{Proposition}
        \newtheorem{lemma}[theorem]{Lemma}
        \newtheorem{corollary}[theorem]{Corollary} 
        \newtheorem{definition}[theorem]{Definition} 
        \newtheorem{remark}[theorem]{Remark}  

\usepackage[shortcuts]{extdash}

\numberwithin{equation}{section}

\def\R{\mathbb{R}}
\def\T{\mathbb{T}}
\def\cP{\mathbb{P}}
\def\div{\mathrm{div}}

\def\cD{\mathcal{D}}
\def\cS{\mathcal{S}}

\def\cH{\mathcal{H}}
\def\cU{\mathcal{U}}

\def\logl{ L\log L}

\def\aL{\mathsf L}
\def\aM{\mathsf M}
\def\aaE{\mathsf E}

\def\sE{\mathsf E}
\def\aS{\mathsf S}
\def\aZ{\mathsf Z}
\def\az{\mathsf z}
\def\aF{\mathsf F}
\def\aG{\mathsf G}
\def\aJ{\mathsf J}
\def\av{\mathsf v}
\def\aw{\mathsf w}

\def\hf{\mathbb f}

\def\Z{{\mathbb Z}}
\def\G{\R^{4d}}

\def\Ga{(\Doa)^2}
\def\TGRE{\mathcal{TGRE}}
\def\cL{\mathcal{L}}
\def\cJ{\mathcal{J}}
\def\cA{\mathcal{A}}
\def\cM{\mathcal{M}}
\def\cE{\mathcal{E}}
\def\supp{\operatorname{supp}}

\newcommand{\defeq }{:=}
\def\d{\partial}
\def\Do{\R^{2d}}
\def\Doa{\Omega\times\R^d}
\def\Dot{\T^d\times\R^d}
\newcommand{\dd}{{\,\rm d}}

\theoremstyle{plain}

\usepackage{todonotes}

\title{On a fuzzy Landau Equation:\ I.\ a variational approach}

\author{Manh Hong Duong
  \and  Zihui He
}



\newcommand{\Addresses}{{
  \bigskip
  \footnotesize

  H.~Duong, \textsc{School of Mathematics, University of Birmingham, UK}\par\nopagebreak
  \textit{E-mail}: \texttt{h.duong@bham.ac.uk}

  \medskip

  Z.~He, \textsc{Fakult\"at f\"ur Mathematik, Universit\"at Bielefeld, Postfach 100131, 33501 Bielefeld, Germany}\par\nopagebreak
  \textit{E-mail}: \texttt{zihui.he@uni-bielefeld.de}

}}
\begin{document}

\maketitle

\begin{abstract}
In this paper, we introduce a fuzzy Landau equation, where particles interact through delocalised collisions. We establish a variational characterisation that recasts the fuzzy Landau equation within the framework of GENERIC systems (General Equations for Non-Equilibrium Reversible-Irreversible Coupling), obtaining an equivalent characterisation of $\cH$-solutions.
\end{abstract}
\tableofcontents

\section{Introduction}
In this paper, we consider the inhomogeneous (kinetic) \textit{fuzzy Landau equation:}
\begin{equation}
\label{Landau-fuz}
    \d_t f+v\cdot\nabla_x f=Q_{\sf fuz}(f,f)
    \quad\text{on}\quad[0,T]\times\Omega\times\R^d,
    \quad f(0,x,v)=f_0(x,v).
\end{equation}
The spatial domain $\Omega$ can be $\T^d$ or $\R^d$ ($d\ge2$). This equation characterises the time evolution of the unknown $f_t(x,v):[0,T]\times\Doa\to\R_+$, describing the distribution of particles in a plasma at time $t$, at position $x$ and with velocity $v$. The initial datum $f_0$ is a given distribution on the phase space. The fuzzy Landau collision term is given by
\begin{equation*}
\begin{aligned}
Q_{\sf fuz}(f,f)
=&\nabla_v\cdot\Big(
\int_{\Doa}\kappa(x-x_*)A(v-v_*)\Pi_{(v-v_*)^\perp}
\big(f_*\nabla_v f-f\nabla_{v_*}f_*\big)\dd x_*\dd v_*
\Big),
\end{aligned}
\end{equation*}
where $f_*=f(x_*,v_*)$ denotes the particle density at position and velocity $(x_*,v_*)$.
The operator $\Pi_{z^\perp}$ denotes the projection onto $z^\perp$:
\begin{align*}
    \Pi_{z^\perp}
    =\operatorname{Id}-\frac{z\otimes z}{|z|^2},
    \quad z\in\R^d\setminus\{0\}.
\end{align*}
We consider the power-law interaction kernel $A(z)=|z|^{2+\gamma}$ with $ \gamma\in(-\gamma_d,1]$, 
where $\gamma_d=\min(d,4)$. 
The parameter $\gamma$ characterises the intermolecular interaction potential: $\gamma>0$ corresponds to the so-called \emph{hard} potentials and $\gamma<0$ to the \emph{soft} potentials, with a further classification of $-2<\gamma<0$ as \emph{moderately soft} potentials and $-d<\gamma\le-2$ as \emph{very soft} potentials. In particular, $\gamma=0$ is known as the \emph{Maxwellian} case, while in dimension three, $\gamma=-3$ corresponds to \emph{Coulomb} interactions. This range excludes the Coulomb endpoint in dimension three. The condition $\gamma>-d$ ensures the local integrability of $|z|^\gamma$, and the condition $\gamma>-4$ is used to justify the weak collision formulation; see Part II \cite{DH25b}.

In the formulation \eqref{Landau-fuz}, we allow delocalised collisions. This means that two particles can collide at different positions and velocities $(x,v)$ and $(x_*,v_*)$, where the interaction between their positions is weighted by the spatial kernel $\kappa$. On $\R^d$, we consider continuous and even spatial kernels $\kappa\sim1$, with positive lower and upper bounds. On $\T^d=\R^d/\Z^d$, we take a fixed $\sigma\in(0,1)$ and 
\begin{align*}
    \kappa(x)\sim\sigma^{-\frac d2}
    \sum_{n\in\Z^d}
    \exp\Big(-\Big\langle\frac{x+n}{\sqrt{\sigma}}\Big\rangle\Big).
\end{align*}
The precise assumptions can be found in Assumption~\ref{ASS:kernel} below.

The fuzzy Landau equation can be viewed as an intermediate model between the spatially homogeneous Landau equation and the classical inhomogeneous Landau equation. For $\kappa=1$, integrating \eqref{Landau-fuz}  over the spatial variable $x$ formally gives the homogeneous Landau equation for the velocity marginal; see Remark~\ref{rem:sigma} below. If the spatial kernel is formally replaced by a Dirac measure, we recover the classical inhomogeneous Landau equation
\begin{equation}
\label{col:Landau-cla}
\begin{aligned}
    &\d_t f+v\cdot\nabla_x f=Q_{\sf class}(f,f),\\
    &Q_{\sf class}(f,f)
    =\nabla_v\cdot
    \int_{\R^d}A(v-v_*)\Pi_{(v-v_*)^\perp}
    \big(f(x,v_*)\nabla_v f(x,v)
    -f(x,v)\nabla_{v_*}f(x,v_*)\big)\dd v_*.
\end{aligned}
\end{equation}

The classical Landau equation, which is an important partial differential equation in kinetic theory, describes the evolution of a particle distribution in a plasma, accounting for both free transport and collisions. The classical Landau collisions in \eqref{col:Landau-cla} are localised in the sense that two particles with velocities $v$ and $v_*\in\R^d$ collide at the same position $x\in\Omega$.

In this article, which is the first part of our series, we show that the fuzzy Landau equation \eqref{Landau-fuz} can be rigorously cast into the so-called \textit{GENERIC framework} via a variational formulation, obtaining an equivalent characterisation of $\cH$-solutions.

The rest of this introduction is organised as follows. In Section~\ref{subsec:GENERIC}, we recall the GENERIC framework. In Section~\ref{subsec:rel-work}, we review related work on the fuzzy Landau equation \eqref{Landau-fuz}. In Section~\ref{subsec:main-thm}, we present and discuss the main result.

\subsection{GENERIC structure}\label{subsec:GENERIC}

The GENERIC \textit{(General Equations for Non-Equilibrium Reversible-Irreversible Coupling)} framework provides a systematic method to derive thermodynamically consistent evolution equations describing complex systems consisting of both Hamiltonian (reversible) dynamics and gradient (irreversible) flows. It was first introduced in the context of complex fluids \cite{GO97a,GO97b}. Over the last 30 years, it has proved to be a versatile modelling tool for a variety of complex systems in physics and engineering, ranging from anisotropic inelastic solids to viscoplastic solids and thermoelastic dissipative materials. We refer the reader to the monographs \cite{Ott05,pavelka2018multiscale} for a detailed exposition of this framework. However, a mathematically rigorous foundation for GENERIC is still underdeveloped. The present work contributes towards this overarching goal.

A GENERIC system describes the evolution of an unknown $\mathsf z$ in a state space $\mathsf Z$ via the equation
\begin{equation}\label{eq:generic}
    \partial_t\az=\aL\dd\aaE+\aM\dd\aS.
\end{equation}
In this equation,
\begin{itemize}
\item The functionals $\aaE,\aS:\aZ\to\R$ are interpreted as energy and entropy functionals, respectively, and $\dd\aaE$, $\dd\aS$ are their differentials.

\item $\aL(\az)$, $\az\in\aZ$, are antisymmetric operators mapping cotangent to tangent vectors and satisfying the Jacobi identity
\begin{equation*}
\{\{\aG_1,\aG_2\}_{\aL},\aG_3\}_{\aL}
+\{\{\aG_2,\aG_3\}_{\aL},\aG_1\}_{\aL}
+\{\{\aG_3,\aG_1\}_{\aL},\aG_2\}_{\aL}=0
\end{equation*}
for all sufficiently smooth functions $\aG_i:\aZ\to\R$, $i=1,2,3$, where the Poisson bracket $\{\cdot,\cdot\}_{\aL}$ is defined by
$$
    \{\aF,\aG\}_{\aL}
    =\dd\aF\cdot\aL\dd\aG.
$$
Here, $\cdot$ signifies the duality pairing of the cotangent and tangent spaces of $\aZ$.
The antisymmetry of $\aL$ means that
$$
    \av\cdot\aL\aw=-\aw\cdot\aL\av
$$
for cotangent vectors $\av,\aw$.

\item $\aM(\az)$, $\az\in\aZ$, are symmetric and positive semidefinite operators mapping cotangent to tangent vectors, in the sense that
$$
    \av\cdot\aM\aw=\aw\cdot\aM\av,
    \quad \av\cdot\aM\av\ge0
$$
for cotangent vectors $\av,\aw$.

\item The following degeneracy conditions are satisfied
\begin{equation}\label{eq:degen}
    \aL(\az)\dd\aS(\az)=0
    \quad\text{and}\quad
    \aM(\az)\dd\aaE(\az)=0
    \quad\text{for all }\az.
\end{equation}
\end{itemize}
Note that pure Hamiltonian systems and pure (dissipative) gradient flow systems are special cases of GENERIC corresponding to $\aM\equiv0$ and $\aL\equiv0$, respectively.

The conditions satisfied by the building blocks $\{\aaE,\aS,\aL,\aM\}$ ensure, formally, that along solutions to \eqref{eq:generic}, the energy $\aaE$ is conserved and the entropy $\aS$ is non-decreasing. In fact, by differentiating the energy functional, we have 
\begin{align*}
\frac{\dd}{\dd t}\aaE(\az_t)
&=\partial_t\az\cdot\dd\aaE
=(\aL\dd\aaE+\aM\dd\aS)\cdot\dd\aaE\\
&=\underbrace{\aL\dd\aaE\cdot\dd\aaE}_{=0}
+\underbrace{\aM\dd\aS\cdot\dd\aaE}_{=0}=0,
\end{align*}
where the first term vanishes due to the anti-symmetry of $\aL$ while the second term is $0$ because of the symmetry of $\aM$ and the second degeneracy condition. 

By a similar computation, we also get
$$
    \frac{\dd}{\dd t}\aS(\az_t)
    =\dd\aS\cdot\aM\dd\aS\ge0.
$$
Thus, the first and second laws of thermodynamics are automatically satisfied for GENERIC systems. This, together with the structures of the energy and entropy functionals, strongly suggests a connection between the GENERIC framework and variational principles. However, while this idea has been rigorously established for pure Hamiltonian systems, see for instance \cite{arnol2013mathematical}, and pure gradient flow systems, see for instance \cite{AGS08}, it is still underdeveloped for full GENERIC systems. This is one of the key challenges in the field; see \cite{Ottinger2018}.

The first author, Peletier and Zimmer \cite{DPZ13} propose a formal variational characterisation: a curve $\az:[0,T]\to\aZ$ is a solution to \eqref{eq:generic} if and only if $J(\az)=0$, where for a give curve $\az$ the functional $J(\az)$ is defined by
\begin{equation}\label{intro:eq:J}
J(\az)
=\aS(\az_0)-\aS(\az_T)
+\frac12\int_0^T
\|\partial_t\az-\aL\dd\aaE\|_{\aM^{-1}}^2\dd t
+\frac12\int_0^T\|\dd\aS\|_{\aM}^2\dd t,
\end{equation}
and the formal bilinear forms associated with $\aM$ and $\aM^{-1}$ are
\begin{equation*}
    (a,b)_{\aM}=a\cdot\aM b
    \quad\text{and}\quad
    (a,b)_{\aM^{-1}}=a\cdot\aM^{-1}b.
\end{equation*}
This variational characterisation is motivated by the following computations for sufficiently regular curves with finite $J$ 
\begin{align*}
J(\az)
&=\aS(\az_0)-\aS(\az_T)
+\frac12\int_0^T
\big(\|\partial_t\az-\aL\dd\aaE\|_{\aM^{-1}}^2
+\|\dd\aS\|_{\aM}^2\big)\dd t\\
&=-\int_0^T\partial_t\az\cdot\dd\aS\dd t
+\frac12\int_0^T
\big(\|\partial_t\az-\aL\dd\aaE\|_{\aM^{-1}}^2
+\|\dd\aS\|_{\aM}^2\big)\dd t\\
&=-\int_0^T
(\partial_t\az-\aL\dd\aaE,\aM\dd\aS)_{\aM^{-1}}\dd t
+\frac12\int_0^T
\big(\|\partial_t\az-\aL\dd\aaE\|_{\aM^{-1}}^2
+\|\dd\aS\|_{\aM}^2\big)\dd t\\
&=\frac12\int_0^T
\|\partial_t\az-\aL\dd\aaE-\aM\dd\aS\|_{\aM^{-1}}^2
\dd t\ge0,
\end{align*}
where the third equality follows from the first degeneracy condition in \eqref{eq:degen}. Thus, $J(\az)$ is nonnegative, and $\az$ is a solution to \eqref{eq:generic} if and only if $J(\az)=0$, or equivalently, $J(\az)\le0$. As shown in \cite{DPZ13}, for the Vlasov--Fokker--Planck equation with prescribed initial data, $J$ is a multiple of the large-deviation rate functional associated with the underlying particle system of weakly interacting Brownian motions.

When $\aaE=0$, the system \eqref{eq:generic} reduces to a gradient flow and the variational characterisation \eqref{intro:eq:J} corresponds to De Giorgi's characterisation of gradient flows as curves of maximal slope; see, for example, \cite{AGS08}. In other words, \eqref{intro:eq:J} can be viewed as an extension of De Giorgi's formulation to GENERIC systems. A natural question is whether this can be made rigorous. This has been achieved for several systems: the kinetic Fokker--Planck equation \cite{DPZ13}, a damped harmonic oscillator \cite{jungel2019two,jungel2021minimizing}, where a discretised version of \eqref{intro:eq:J} is used, and, more recently, a fuzzy Boltzmann equation \cite{EH25}. In the following subsection, we discuss previous work on the Landau equation.

\subsection{Related work}
\label{subsec:rel-work}

In Appendix~\ref{app:A}, we provide a formal calculation showing that the fuzzy Landau equation \eqref{Landau-fuz} can be cast as a GENERIC system by directly constructing the building blocks $\{\aM,\aL,\aaE,\aS\}$ and verifying the required conditions in a Schwartz setting. This formal calculation also applies to the classical Landau equation \eqref{col:Landau-cla}, with the spatial kernel replaced by a Dirac measure.

However, rigorously proving the GENERIC structure of the classical inhomogeneous Landau equation \eqref{col:Landau-cla} is a challenging problem. This is because the quadratic dependence on position of the classical collision operator $Q_{\sf class}(f,f)$ through the product $f(x,v)f(x,v_*)$ is not controlled by $L^1(\Doa)$ bounds on the density. Regarding solvability within the $L^1$ framework, renormalised solutions have been studied in \cite{Vil96}, among others.

We introduce delocalised collisions into the fuzzy Landau equation \eqref{Landau-fuz}, which avoids this particular difficulty since
\begin{equation}
\label{intro:apri}
    \|f(x,v)f(x_*,v_*)\|_{L^1((\Doa)^2)}
    \le\|f(x,v)\|_{L^1(\Doa)}^2.
\end{equation}
This enables us to rigorously establish the GENERIC structure of the fuzzy Landau equation.

Another closely related model is the \emph{spatially homogeneous} Landau equation
\begin{equation}
\label{Landau-homo}
    \partial_t f(v)=Q_{\sf homo}(f,f)(v),
    \quad f_t(v):\R^d\to\R_+,
\end{equation}
with the homogeneous collision term $Q_{\sf homo}$ given in \eqref{Qhomo} below. Recently Carrillo, Delgadino, Desvillettes and Wu \cite{carrillo2024landau} showed that \eqref{Landau-homo} can be interpreted as a gradient flow of the Boltzmann entropy $\cH(f)=\int_{\R^d}f\log f\dd v$ with respect to a Wasserstein-type metric.

Weak $L^1$ solutions of the homogeneous Landau equation \eqref{Landau-homo} have been studied intensively for hard, Maxwellian and soft potentials \cite{Vil98,DV00a,DV00b,Wu14,ALL15}. As a consequence of the a priori estimate \eqref{intro:apri}, the fuzzy Landau equation \eqref{Landau-fuz} is structurally similar to the homogeneous Landau equation \eqref{Landau-homo}. In Part II of this series \cite{DH25b}, we build on existing results for the homogeneous Landau equation to study $\cH$-solutions of the fuzzy Landau equation \eqref{Landau-fuz}.

The Landau equation can also be derived from the grazing collision limit of the Boltzmann equation, where collisions with small angular deviations play a significant role \cite{Vil98b,He14}. The gradient flow structure of the spatially homogeneous Boltzmann equation was studied in \cite{erbar2023gradient}, and a GENERIC structure for the delocalised fuzzy Boltzmann equation was shown by Erbar and the second author in \cite{EH25}. Concerning the grazing collision limit of the Boltzmann equation, a gradient flow perspective for the spatially homogeneous case was discussed in \cite{carrillo2022boltzmann}. Motivated by these results, in Part III \cite{DGH25}, we study the grazing collision limit from the fuzzy Boltzmann equation to the fuzzy Landau equation using their respective GENERIC variational characterisations.

Delocalised collision models have previously been studied in other contexts, such as the Enskog equation for high-density particle systems \cite{Ark86,AC89} and more general delocalised kinetic models \cite{FZF24}, as well as from an optimal transport perspective \cite{GJF25}. Finally, related \emph{discrete} variational formulations, akin to the celebrated Jordan--Kinderlehrer--Otto \cite{JKO98} scheme, for non-gradient flow systems were studied in \cite{gigli2012variational,duong2014conservative,duong2018fundamental,adams2022operator,duong2022entropic}, and extensions of symplectic integrators for Hamiltonian systems to GENERIC systems were discussed in \cite{ottinger2018generic,shang2020structure}.

\subsection{Variational characterisation}\label{subsec:main-thm}

We formally have the following conservation laws for the fuzzy Landau equation \eqref{Landau-fuz}
\begin{align*}
\text{mass}\quad&
\int_{\Doa}f_t\dd x\dd v
=\int_{\Doa}f_0\dd x\dd v,\\
\text{momentum}\quad&
\int_{\Doa}f_tv\dd x\dd v
=\int_{\Doa}f_0v\dd x\dd v,\\
\text{energy}\quad&
\int_{\Doa}f_t|v|^2\dd x\dd v
=\int_{\Doa}f_0|v|^2\dd x\dd v.
\end{align*}
We define the Boltzmann entropy functional as $\cH(f)=\int_{\Doa}f\log f\dd x\dd v$. Formally, the following entropy identity holds:
\begin{align*}
    \cH(f_T)-\cH(f_0)
    =-\int_0^T\cD(f_t)\dd t\le0,
\end{align*}
where the entropy dissipation is given by
\begin{equation}
\label{entropy dissipation}
    \cD(f)=\frac12\int_{(\Doa)^2}
    \kappa ff_*|\widetilde\nabla\log f|^2\dd\eta\ge0.
\end{equation}
Throughout the paper, we use the shortened notation
$\dd\eta=\dd x_*\dd v_*\dd x\dd v$
for the product measure on $\Ga$, and $\widetilde\nabla$ is the fuzzy Landau gradient operator defined below.

A rigorous discussion of moment bounds and entropy inequalities for the constructed solutions is provided in Part II \cite{DH25b}.

As previously mentioned, the fuzzy Landau equation \eqref{Landau-fuz} can be formulated as a GENERIC system using the building blocks $\{\aM,\aL,\aaE,\aS\}$, with the verification provided in Appendix~\ref{app:A}. The energy and entropy are given by
\begin{align*}
    &\aaE(f)=\int_{\Doa}\frac{|v|^2}{2}f\dd x\dd v
    \quad\text{and}\quad
    \aS(f)=-\cH(f),
\end{align*}
and the operators $\aL$ and $\aM$ at $f\in\aZ$ are given by
\begin{align*}
    &\aM(f)g
    =-\frac12\widetilde\nabla\cdot
    \Big(\kappa ff_*\widetilde\nabla g\Big)
    \quad\text{and}\\
    &\aL(f)g=-\nabla\cdot(f\aJ\nabla g),
    \quad
    \aJ=\begin{pmatrix}
    0&\mathsf{id}_d\\
    -\mathsf{id}_d&0
    \end{pmatrix}.
\end{align*}
Here $\nabla$ denotes the gradient in $(x,v)$.
Following \cite{carrillo2024landau}, we define the fuzzy Landau gradient by
\begin{equation}
\label{gradient}
    \widetilde\nabla f(x,v)
    =\sqrt{A}\Pi_{(v-v_*)^\perp}
    \big(\nabla_v f(x,v)-\nabla_{v_*}f(x_*,v_*)\big),
\end{equation}
where $A=A(v-v_*)$.
The corresponding divergence $\widetilde\nabla\cdot$ is defined via integration by parts:
\begin{equation}
\label{tilde:IP}
    \int_{\Ga}\widetilde\nabla g\cdot B\dd\eta
    =-\int_{\Doa}g\widetilde\nabla\cdot B\dd v\dd x
\end{equation}
for sufficiently smooth $B:\Ga\to\R^d$ and $g:\Doa\to\R$ with suitable decay.
From this, we obtain
\begin{equation*}
    \widetilde\nabla\cdot B
    =\nabla_v\cdot\Big(
    \int_{\Doa}\sqrt{A}\Pi_{(v-v_*)^\perp}
    \big(B(x,x_*,v,v_*)-B(x_*,x,v_*,v)\big)
    \dd x_*\dd v_*
    \Big).
\end{equation*}

In this paper, we rigorously establish the variational characterisation \eqref{intro:eq:J} for the GENERIC structure of the fuzzy Landau equation \eqref{Landau-fuz}, using $\cH$ instead of $\aS$. We note that the Boltzmann entropy $\cH$ defined here, as often in the literature on gradient flows, is the negative of the physical entropy $\aS$ in the GENERIC framework; thus it is non-increasing along solutions.

To state the main result, we introduce the transport grazing rate equation
\begin{equation}
\label{intro:TGRE}
    \d_t f+v\cdot\nabla_x f
    +\frac12\widetilde\nabla\cdot U_t=0,
\end{equation}
where $U:[0,T]\times\Ga\to\R^d$ denotes the grazing rate, following the terminology of the homogeneous case \cite{carrillo2024landau}. Notice that, when $U_t=-\kappa ff_*\widetilde\nabla\log f$, equation \eqref{intro:TGRE} recovers the fuzzy Landau equation \eqref{Landau-fuz}. We also define the following curve action associated with \eqref{intro:TGRE}:
\begin{align*}
    \cA(f,U)
    =\frac12\int_{\Ga}\frac{|U|^2}{ff_*\kappa}\dd\eta.
\end{align*}
See Definition~\ref{def: action} and Lemma~\ref{lem:U:density} below for its precise definition and characterisation for measure-valued grazing rates.

\medskip

The following theorem is the main result of this paper.
\begin{theorem}\label{thm:main}
Let $A$ and $\kappa$ satisfy Assumption~\ref{ASS:kernel}. Let $(f_t\cL,\cU_t)\in\TGRE_T$ be a solution pair of the transport grazing rate equation as defined in Section~\ref{TGE}.
Suppose that the curve of probability densities on $\Doa$, $(f_t)_{t\in[0,T]}$,  satisfies Assumption~\ref{ass:curve}, with
\begin{align*}
    \int_{\Doa}
    \big(\langle x\rangle^2+\langle v\rangle^{2+|\gamma|}\big)
    f_t\dd x\dd v<+\infty
    \quad\text{for all }t\in[0,T],
\end{align*}
and the initial entropy $|\cH(f_0)|<+\infty$. On $\Dot$, the spatial moment condition is omitted.

Then we have
\begin{equation*}
    \cJ_T(f,\cU)
    :=\cH(f_T)-\cH(f_0)
    +\frac12\int_0^T\cD(f_t)\dd t
    +\frac12\int_0^T\cA(f_t,\cU_t)\dd t
    \ge0.
\end{equation*}
Moreover, $\cJ_T(f,\cU)=0$ if and only if $f$ is an $\cH$-solution of the fuzzy Landau equation \eqref{Landau-fuz} and
\begin{align*}
    \cU_t=-\kappa ff_*\widetilde\nabla\log f\dd\eta
\end{align*}
for almost every $t\in[0,T]$.
\end{theorem}

The definition of $\cH$-solutions is provided in Definition~\ref{def:H-sol}, following the framework for homogeneous Landau equations introduced by \textcite{Vil98b}. The proof of Theorem~\ref{thm:main}, given in Section~\ref{sec:vc}, amounts to rigorously verifying a chain rule for the transport grazing rate equation \eqref{intro:TGRE}. We achieve this using regularisation techniques and error estimates. Our analysis requires significant technical modifications to the homogeneous argument in \cite{carrillo2024landau} to account for the spatial variable $x$ and the transport term $v\cdot\nabla_x f(x,v)$. 


In the related work \cite{DH26}, instead of the gradient operator defined in \eqref{gradient}, we introduce a cross-product Landau gradient whose lifted operator commutes with Gaussian mollification in the velocity variables. For the regularised interaction kernels \eqref{A-2}--\eqref{A:Rd2:0}, this formulation yields a variational characterisation for the spatially homogeneous Landau equation \eqref{Landau-homo}  under weighted $L^1$ assumptions, without the additional $L^p$ bounds on the curves required by the earlier approach in \cite{carrillo2024landau}. The authors also discuss the corresponding formulation for fuzzy Landau equations. In particular, the Maxwellian kernel is included. We refer to Remark~\ref{rem:cross} below for further discussion.

\medskip

The rest of the paper is organised as follows. In Section~\ref{sec:pre}, we introduce the necessary notation and discuss curves satisfying Assumption~\ref{ass:curve}. In Section~\ref{sec:vc}, we perform regularisations and derive error estimates to rigorously verify the chain rule and prove our main result, Theorem~\ref{thm:main}. In Section~\ref{sec: conclusion}, we discuss different representations of the GENERIC structure of fuzzy Landau equations and directions for future work. Finally, the formal GENERIC formulation is verified in Appendix~\ref{app:A}.

This article is the first in a series of works on the fuzzy Landau equation. The second paper \cite{DH25b} is devoted to existence and regularity results. The third paper \cite{DGH25} studies the grazing collision limit from fuzzy Boltzmann equations to fuzzy Landau equations using the variational method introduced in this paper.

\subsection*{Acknowledgements}
M. H. D is funded by an EPSRC Standard Grant EP/Y008561/1. Z.~H. is funded by the Deutsche Forschungsgemeinschaft (DFG, German Research Foundation) – Project-ID 317210226 – SFB 1283.

The authors thank Matthias Erbar for useful comments on this manuscript.

\section{Preliminaries}\label{sec:pre}

We will present the kernel assumptions in Section~\ref{subsec:kernel}. We will introduce the notation in Section~\ref{subsec:notation}. Assumption~\ref{ass:curve} on the curves in Theorem~\ref{thm:main} will be discussed in Section~\ref{subsec:sol}, and related estimates will be provided in Section~\ref{subsec:est}.

Throughout this article, for simplicity, we write the domain as $\Do$, which denotes either $\R^{2d}$ or $\Dot$, unless otherwise stated.

Throughout this work, $C>0$ denotes a positive constant, which may vary from line to line.

\subsection{Kernel assumptions}
\label{subsec:kernel}

The spatial kernel $\kappa$ and the interaction kernel $A$
satisfy the following assumptions.

\begin{assumption}[Kernel assumptions]
\label{ASS:kernel}
The kernel $A:\R_+\to\R_+$ is given by
\begin{equation}
\label{A-1}
    A(|v-v_*|)=|v-v_*|^{2+\gamma}
    \quad\text{with}\quad\gamma\in(-\gamma_d,1].
\end{equation}

In the case of $\R^d$, the spatial kernel $\kappa:\R^d\to\R_+$ is continuous and even, and satisfies
\begin{equation}
\label{k-1}
    C^{-1}\le\kappa(x)\le C
\end{equation}
for some constant $C\ge1$.

In the case of $\T^d=\R^d/\Z^d$, for any fixed $\sigma\in(0,1)$, the spatial kernel $\kappa:\T^d\to\R_+$ is given by
\begin{equation}
\label{k-2}
    \kappa(x)
    =C\sigma^{-\frac d2}
    \sum_{n\in\Z^d}
    \exp\Big(-\Big\langle\frac{x+n}{\sqrt{\sigma}}\Big\rangle\Big),
\end{equation}
where $C>0$ is independent of $\sigma$ and is chosen such that
$\|\kappa\|_{L^1(\T^d)}=1$.
\end{assumption}

\begin{remark}[A cross-product Landau gradient]\label{rem:cross}
In \cite{DH26}, the authors consider a cross-product Landau gradient, which in dimension three acts on $f$ as
\begin{align*}
    (v-v_*)\times
    \big(\nabla_v f(x,v)-\nabla_{v_*}f(x_*,v_*)\big).
\end{align*}
In higher dimensions, the cross product is replaced by its skew-symmetric matrix analogue. We consider the kernels $A$ of the following form:
\begin{equation}
\label{A-2}
    A(|v-v_*|)=A_0(|v-v_*|)|v-v_*|^2,
    \quad\forall\,v,\,v_*\in\R^d,
\end{equation}
where $A_0$ is even and continuous, and
\begin{equation}
\label{A:Rd2:0}
    A_0(|v-v_*|)\sim\langle v-v_*\rangle^\gamma,
    \quad\gamma\in(-\infty,1].
\end{equation}
The lifted cross-product operator commutes with Gaussian mollification in the velocity variables, while $A_0$ and $A_0^{-1}$ satisfy uniform convolution bounds.

Using the above cross-product gradient, a variational characterisation analogous to Theorem~\ref{thm:main} holds without the additional $L^p$ bounds on the curves, under the moment and entropy assumptions stated there. In particular, Assumption~\ref{ass:curve} on the curves is replaced by
\begin{equation}
\label{curve-2}
    f\in L^1\big([0,T];L^1_{2,2+\gamma_+}(\Do)\big).    
\end{equation}
This approach also allows spatially decaying kernels outside Assumption~\ref{ASS:kernel}, such as, for any fixed $\sigma\in(0,1)$,
\begin{align}
\label{k-3}
    \kappa^\sigma(x)
    =C\sigma^{-d/2}\exp(-\langle x/\sqrt{\sigma}\rangle),
    \quad x\in\R^d,
\end{align}
where $C>0$ is chosen such that
$\|\kappa^\sigma\|_{L^1(\R^d)}=1$.
We refer to \cite{DH26} for more details.

Related kinetic and spatial kernels were considered for the fuzzy Boltzmann equation in \cite{EH25}. The grazing collision limit from non-cutoff fuzzy Boltzmann equations to fuzzy Landau equations is studied in \cite{DGH25}, under the kernel and scaling assumptions stated there.
\end{remark}

\begin{remark}\label{rem:sigma}
The fuzzy Landau equation \eqref{Landau-fuz} can be seen as an intermediate model between the homogeneous \eqref{Landau:homo} and inhomogeneous \eqref{col:Landau-cla} Landau equations.

For $\kappa=1$, we integrate the fuzzy Landau equation \eqref{Landau-fuz} over the spatial variable $x$, which at least formally gives the homogeneous Landau equation. Indeed, we define $F(v)=\int_x f(x,v)\dd x$. Then $F(v)$ satisfies
\begin{equation}
\label{Landau:homo}
    \begin{cases}
    \d_t F=Q_{\sf homo}(F,F),\\
    F|_{t=0}=\int_{\R^d} f_0(x,v)\dd x,
    \end{cases}
\end{equation}
where
\begin{equation}
\label{Qhomo}
Q_{\sf homo}(F,F)
=\nabla_v\cdot\Big(
\int_{\R^d}A(v-v_*)\Pi_{(v-v_*)^\perp}
\big(F(v_*)\nabla_v F(v)-F(v)\nabla_{v_*}F(v_*)\big)\dd v_*
\Big).
\end{equation}

For $\kappa$ as in \eqref{k-2}, the measures $\kappa(x)\dd x$ converge weakly to a Dirac measure as $\sigma\to0$. Formally, the fuzzy Landau equation \eqref{Landau-fuz} then converges to the classical inhomogeneous Landau equation \eqref{col:Landau-cla}.
\end{remark}

\begin{remark}\label{rmk:kappa}
We note that in the cases of \eqref{k-1} and \eqref{k-2}, the spatial kernels have positive lower and upper bounds:
\begin{align*}
 C^{-1}\le\kappa(x)\le C\quad\text{on $\R^d$ or $\T^d$}
\end{align*}
for some constant $C\ge1$, which may depend on $\sigma$ in the torus case. Notice that $\kappa=1$ is included in the case of $\R^d$.

On $\R^d$, we define $M(z)=C\exp(-\langle z\rangle)$, where we choose the constant $C>0$ such that $\|M\|_{L^1(\R^d)}=1$.
We define the rescaled mollifiers with parameter $\beta\in(0,1)$ by
\begin{equation}
\label{sec2:def:M:beta}
    M_\beta(z)=C\beta^{-d}\exp(-\langle z/\beta\rangle).
\end{equation}
Analogously, we periodise these mollifiers on $\T^d$:
\begin{align*}
    M_\beta(z)
=\sum_{k\in\mathbb Z^d}M_\beta(z+k)
=C\beta^{-d}\sum_{k\in\mathbb Z^d}
\exp \big(-\left\langle {(z+k)}/{\beta}\right\rangle\big).
\end{align*}
We note that the positive lower and upper bounds on $\kappa$ ensure that
\begin{equation}
\label{kernel2-2}
\begin{aligned}
    \kappa*M_\beta\lesssim \kappa
    \quad\text{and}\quad
    \kappa^{-1}*M_\beta\lesssim \kappa^{-1},
\end{aligned}
\end{equation}
uniformly in $\beta\in(0,1)$.
For spatially decaying kernels such as \eqref{k-3}, analogous bounds with Gaussian mollifiers follow from the proof of \cite[Lemma 2.1]{EH25}.

However, the pair of bounds in \eqref{kernel2-2} fails for the kinetic kernel \eqref{A-1} when $\gamma\neq-2$, since either $A$ or $A^{-1}$ vanishes as $v-v_*\to0$. The analogous bounds hold for $A_0$ in \eqref{A:Rd2:0}. This is one of the key points in eliminating the additional $L^p$ bounds on curves in \cite{DH26}.

\end{remark}

We will also use Peetre's inequality
\begin{equation*}
    \frac{\langle x\rangle^p}{\langle y\rangle^p}
    \le 2^{|p|/2}\langle x-y\rangle^{|p|}
    \quad\forall\,x,\,y\in\R^d,\quad p\in\R,
\end{equation*}
which implies
\begin{equation}
\label{poly:M:bdd}
    \big(\langle\cdot\rangle^p*M_\beta\big)(z)
    \le C\langle z\rangle^p,
\end{equation}
where $C>0$ depends only on $|p|$ and
$\|\langle\cdot\rangle^{|p|}M\|_{L^1(\R^d)}$ and is independent of $\beta\in(0,1)$.

\subsection{Notations}
\label{subsec:notation}
We use the notation $B^R_x:=\{x\in\R^d\mid |x|\le R\}$,
and define $B_v^R$ and $B_{x,v}^R$ analogously.

Let $a,\,b\in\R$. We define
\begin{align*}
    L^1_{a,b}(\Do)
    =\{f\in L^1(\Do)\mid
    \|f\|_{L^1_{a,b}(\Do)}<+\infty\},
\end{align*}
with the norm
\begin{align*}
    \|f\|_{L^1_{a,b}(\Do)}
    :=\int_{\Do}
    \big(\langle x\rangle^a+\langle v\rangle^b\big)
    |f(x,v)|\dd x\dd v,
\end{align*}
where we recall that $\langle z\rangle:=\sqrt{1+|z|^2}$ denotes the Japanese bracket.

Let $p,\,q\in[1,\infty)$. The function space $L^p_vL^q_x(\Do)$ consists of all measurable functions $f$ such that
\begin{align*}
    \|f\|_{L^p_vL^q_x}
    \defeq\Big(
    \int_{\R^d}\Big(\int_x|f(x,v)|^q\dd x\Big)^{\frac pq}\dd v
    \Big)^{\frac1p}<+\infty,
\end{align*}
where the spatial integral is over $\R^d$ or $\T^d$, depending on the spatial domain.
Similarly, we define the function spaces $L^q_xL^p_v(\Do)$.

We define $\cP(\Do)$ as the space of Borel probability measures on $\Do$, endowed with the weak topology induced by duality with bounded continuous functions in $C_b(\Do)$.

For any $a,\,b\in\R$, we define the following subspace of $\cP(\Do)$:
\begin{align*}
    \cP_{a,b}(\Do)
    =\{\mu\in\cP(\Do)\mid\cE_{a,b}(\mu)<+\infty\},
\end{align*}
where
\begin{align*}
    \cE_{a,b}(\mu)
    \defeq\int_{\Do}
    \big(\langle x\rangle^a+\langle v\rangle^b\big)\dd\mu.
\end{align*}
For nonnegative functions $f:\Do\to\R$, we also write $\cE_{a,b}(f)= \|f\|_{L^1_{a,b}(\Do)}$.

On $\Dot$, spatial moment weights can be omitted.

Next we recall the Boltzmann entropy functional for measures $\mu\in\cP_{2,2}(\Do)$. If $\mu=f\cL$ is absolutely continuous with respect to Lebesgue measure $\cL$ with density $f$, we define
\begin{align*}
    \cH(\mu)=\int_{\Do}f\log f\dd x\dd v.
\end{align*}
Otherwise, we define $\cH(\mu)=+\infty$. In the case of $\mu=f\cL$, we also write $\cH(f)$.

An important property used in the subsequent analysis is the following lower semicontinuity of the entropy: if $\mu_n\to\mu$ weakly in $\cP(\Do)$ and
$\cE_{2,2}(\mu_n)\to\cE_{2,2}(\mu)<+\infty$ as $n\to\infty$, then
\begin{align*}
    \cH(\mu)\le\liminf_{n\to\infty}\cH(\mu_n).
\end{align*}

\subsection{Assumptions on curves}
\label{subsec:sol}

The fuzzy Landau equation \eqref{Landau-fuz} can be written as
\begin{equation}
    (\d_t+v\cdot\nabla_x)f
    =\frac12\widetilde\nabla\cdot
    \big(\kappa ff_*\widetilde\nabla\log f\big),
\end{equation}
where we recall that the gradient operator $\widetilde{\nabla}$ is defined in \eqref{gradient} and the divergence operator $\widetilde\nabla\cdot$ is defined via the integration by parts formula \eqref{tilde:IP}:
\begin{equation*}
    \int_{\G}\widetilde\nabla g\cdot B\dd\eta
    =-\int_{\Do}g\widetilde\nabla\cdot B\dd v\dd x
\end{equation*}
for sufficiently smooth $B:\Ga\to\R^d$ and $g:\Do\to\R$ with suitable decay.

We follow \cite{Vil98b} to define an $\cH$-solution of the fuzzy Landau equation \eqref{Landau-fuz}.
\begin{definition}[$\cH$-solution]\label{def:H-sol}
Let $f_0\in L^1_{2,2}(\Do)\cap\logl(\Do;\R_+)$.
We say that a nonnegative function
\begin{align*}
    f\in C([0,T];\cD'(\Do))
    \cap L^\infty([0,T];L^1_{2,2}(\Do)\cap\logl(\Do))
\end{align*}
is an $\cH$-solution of \eqref{Landau-fuz} if the entropy
dissipation is time-integrable,
\begin{align*}
    \int_0^T\cD(f_t)\dd t
    =\frac12\int_0^T\int_{\G}
    \kappa ff_*|\widetilde\nabla\log f|^2
    \dd\eta\dd t<+\infty,
\end{align*}
and the following identity holds:
\begin{equation}
\label{FL:H:ab}
\begin{aligned}
&\int_{\Do}f_0\varphi(0)\dd x\dd v
+\int_0^T\int_{\Do}
f(\d_t\varphi+v\cdot\nabla_x\varphi)\dd x\dd v\dd t\\
&=\frac12\int_0^T\int_{\G}
\kappa ff_*\widetilde\nabla\varphi\cdot
\widetilde\nabla\log f\dd\eta\dd t
\end{aligned}
\end{equation}
for all $\varphi\in C^\infty_c([0,T)\times\Do)$,
with all terms finite.
\end{definition}

Notice that, on the domain $\Dot$, we do not need an extra moment assumption on $x$.

In the main theorem, Theorem~\ref{thm:main}, we consider the curves $\mu_t=f_t\cL$ satisfying the following assumptions.
\begin{assumption}
\label{ass:curve}
For $(f_t\cL)_{t\in[0,T]}\subset\cP(\Do)$, we assume
\begin{align*}
    f\in L^1\big([0,T];L^1_{2,0}(\Do)\big)
    \cap L^\infty\big([0,T];L^1_{0,2+|\gamma|}(\Do)\big).
\end{align*}
In addition, in the soft potential case $\gamma\in(-\gamma_d,0)$, we assume
\begin{align*}
    \langle v\rangle^{2+|\gamma|}f
    \in L^\infty([0,T];L^p_vL^1_x).
\end{align*}
In the hard potential case $\gamma\in(0,1]$, we assume
\begin{align*}
     f
    \in L^\infty([0,T];L^p_vL^1_x).
\end{align*}
In both cases, $p\in\big(\frac{d}{d-|\gamma|},+\infty\big)$.

In the Maxwellian case $\gamma=0$, we assume
$\sup_{t\in[0,T]}\cH(f_t)<+\infty$.

\end{assumption}

\begin{remark}\label{hard-potential:p}
One can replace $L^p_vL^1_x$ by $L^1_xL^p_v$ in these assumptions. Indeed, Minkowski's inequality gives
$\|f\|_{L^p_vL^1_x}\le\|f\|_{L^1_xL^p_v}$, and the same inequality holds with the velocity weights above. Thus, the corresponding bounds required in Lemma~\ref{bdd:conv} also hold under these stronger mixed-norm assumptions.

On the torus, one can also replace these mixed-norm spaces by $L^p(\Dot)$, with the same velocity weights. Indeed, since $\T^d=\R^d/\Z^d$ has unit volume, Hölder's inequality gives
\begin{align*}
    \|f\|_{L^p_vL^1_x}
    \le\|f\|_{L^1_xL^p_v}
    \le\|f\|_{L^p(\Dot)},
\end{align*}
and the same inequalities hold with the velocity weights above.
\end{remark}

\begin{remark}
\begin{itemize}
\item
Existence results for $\cH$-solutions are established in
Part II \cite{DH25b} of this series, where we also study the propagation
of moments and $L^p$-norms.
Under additional assumptions on the initial value, these
estimates verify Assumption~\ref{ass:curve} for moderately
soft potentials on both the torus and the whole space, and
for hard potentials with $\kappa\equiv1$.
For very soft potentials with $\kappa\equiv1$,
local weighted a priori estimates are also obtained for sufficiently regular
solutions with sufficiently small initial weighted norm
on a prescribed interval $[0,T]$.
A general global weighted $L^p$-estimate for very soft potentials
without these smallness and regularity restrictions remains to be established.
We refer to \cite[Remark 4.6]{DH25b} for a detailed
verification of the curve assumptions used here.

\item The entropy dissipation bound is natural since we will show in Part II \cite[Theorem 1.1]{DH25b} that the following entropy inequality holds:
\begin{align}
\label{H:ineq}
    \cH(f_T)-\cH(f_0)+\int_0^T\cD(f_t)\dd t\le0.
\end{align}
For $\cH$-solutions satisfying Assumption~\ref{ass:curve} and the remaining hypotheses of Theorem~\ref{thm:main}, this theorem implies that \eqref{H:ineq} holds with equality. Hence, we obtain the entropy identity
\begin{align*}
    \cH(f_T)-\cH(f_0)+\int_0^T\cD(f_t)\dd t=0.
\end{align*}
For the Maxwellian and regularised hard potential kernels discussed in Remark~\ref{rem:cross}, the corresponding cross-product variational characterisation likewise yields entropy equality whenever the solutions constructed in \cite{DH25b} satisfy \eqref{curve-2} and the remaining hypotheses of that characterisation.
\end{itemize}
\end{remark}

\subsection{Entropy dissipation and Fisher information}
\label{subsec:est}

In this subsection, we consider curves satisfying Assumption~\ref{ass:curve}, with the kernels $A$ and $\kappa$ satisfying Assumption~\ref{ASS:kernel}.

In the subsequent analysis, we will use the following identity:
\begin{equation*}
    |x|^2(y\cdot\Pi_{x^\perp}y)
    =\frac12\sum_{i,j=1}^d|x_iy_j-x_jy_i|^2
    \quad\forall\,x\in\R^d\setminus\{0\},\,y\in\R^d,
\end{equation*}
which can be verified by direct computation:
\begin{align*}
|x|^2(y\cdot\Pi_{x^\perp}y)
&=|x|^2\Big[y-\frac{x\otimes x}{|x|^2}y\Big]\cdot y\\
&=|x|^2|y|^2-(x\otimes x)y\cdot y\\
&=|x|^2|y|^2-\sum_{i,j}(x\otimes x)_{ij}y_iy_j\\
&=\Big(\sum_i x_i^2\Big)\Big(\sum_j y_j^2\Big)
-\sum_{i,j}x_ix_jy_iy_j\\
&=\Big(\sum_i x_i^2\Big)\Big(\sum_j y_j^2\Big)
-\Big(\sum_i x_iy_i\Big)^2\\
&=\frac12\sum_{i,j}|x_iy_j-x_jy_i|^2.
\end{align*}
Notice that in dimension three, this reads
$|x|^2(y\cdot\Pi_{x^\perp}y)=|x\times y|^2$.
To simplify the notation, we write
\begin{align*}
    [x,y]_{ij}=x_iy_j-x_jy_i,
    \quad i,j=1,\dots,d.
\end{align*}
Then the entropy dissipation \eqref{entropy dissipation} can be written as
\begin{equation}
\label{dissipation:cross}
\begin{aligned}
\cD(f)
=&\frac14\sum_{i,j=1}^d
\int_{\G}\kappa ff_*|v-v_*|^\gamma
\big|[v-v_*,\nabla_v\log f-\nabla_{v_*}\log f_*]_{ij}\big|^2
\dd\eta.
\end{aligned}
\end{equation}
Indeed, we have
\begin{equation}
\label{proof:cross}
\begin{aligned}
&|\widetilde\nabla\log f|^2\\
=&{}|v-v_*|^\gamma|v-v_*|^2
\big(\nabla_v\log f-\nabla_{v_*}\log f_*\big)
\cdot\Pi_{(v-v_*)^\perp}
\big(\nabla_v\log f-\nabla_{v_*}\log f_*\big)\\
=&{}\frac{|v-v_*|^\gamma}{2}
\sum_{i,j=1}^d
\big|[v-v_*,\nabla_v\log f-\nabla_{v_*}\log f_*]_{ij}\big|^2.
\end{aligned}
\end{equation}

We define the weighted Fisher and cross-Fisher information, respectively, by
\begin{equation*}
\label{FF}
\int_{\Do}f\langle v\rangle^\gamma
\Big|\frac{\nabla_v f}{f}\Big|^2\dd x\dd v
\quad\text{and}\quad
\sum_{i,j=1}^d\int_{\Do}f\langle v\rangle^\gamma
\Big|\Big[v,\frac{\nabla_v f}{f}\Big]_{ij}\Big|^2\dd x\dd v.
\end{equation*}

We will follow the arguments of \cite{Des15,Des16,Des20} for homogeneous Landau equations to show that the weighted Fisher and cross-Fisher information can be bounded by the entropy dissipation. We first discuss the hydrodynamic bounds in Lemma~\ref{lem:hydro}, then show the weighted Fisher information bounds in Lemma~\ref{lem:Des}.

We define
$$
    \hf(x,v)=\int_{x_*}f(x_*,v)\kappa(x-x_*)\dd x_*.
$$
We define
\begin{align*}
    m(f):=\int_v f(x,v)\dd v,\quad
    E(f):=\int_v\langle v\rangle^2f(x,v)\dd v,\quad
    H(f):=\int_v(\log f)_+f(x,v)\dd v.
\end{align*}

\begin{lemma}[Hydrodynamic bounds]\label{lem:hydro}
Let $\kappa$ satisfy Assumption~\ref{ASS:kernel}. Let 
$p\in(1,\infty)$. Let $f\in L^1_{2,2}(\Do)\cap\logl(\Do;\R_+)$ such that
$\|\langle v\rangle^2f\|_{L^1_{x,v}}\le E_0$ and
$\|(\log f)_+f\|_{L^1_{x,v}}\le H_0$
for some constants $E_0,\,H_0>0$. We have
\begin{align*}
    E(\hf)\le\|\kappa\|_{L^\infty}E_0
    \quad\text{and}\quad
    \sup_x\|\hf(x,\cdot)\|_{L^p_v}
    \le\|\kappa\|_{L^\infty}\|f\|_{L^p_vL^1_x}.
\end{align*}
Moreover, the following estimates hold:
\begin{itemize}
\item In the case of $\kappa\sim1$ on $\R^d$, we have
\begin{gather}
\label{H:k-1}
H(\hf)\lesssim_\kappa H_0+\|f\|_{L^1_{2,2}},
\quad
m(\hf)\gtrsim_\kappa \|f\|_{L^1_{x,v}}.
\end{gather}
\item In the case of $x\in\T^d$, with the normalisation
in Assumption~\ref{ASS:kernel}, we have
\begin{equation}
    E(\hf)\le\|\kappa\|_{L^\infty}E_0,\quad
    H(\hf)\le\|\kappa\|_{L^\infty}H_0.
\end{equation}
Moreover, if $\inf_x\kappa>0$, we have
\begin{equation}
    m(\hf)\gtrsim_\kappa\|f\|_{L^1_{x,v}}.
\end{equation}
\end{itemize}
\end{lemma}

\begin{proof}
When $\kappa$ is normalised on $\T^d$, by using the
convexity of $x\mapsto x(\log x)_+$ and Jensen's inequality,
we have
\begin{align*}
    H(\hf)
    &\le\int_v\int_{\T^d}\kappa(x-x_*)(\log f(x_*,v))_+
    f(x_*,v)\dd x_*\dd v
    \le\|\kappa\|_{L^\infty}H_0.
\end{align*}

In the case of $\kappa\sim1$ on $\R^d$, since
$0\le\kappa\le\kappa_0$, by monotonicity, we have
\begin{align*}
0\le H(\hf)
&\le \kappa_0\int_v\Big(\int_xf\dd x\Big)
\Big(\log \kappa_0+\log \int_xf\dd x\Big)_+\dd v\\
&\lesssim_\kappa \|f\|_{L^1_{x,v}}
+H\Big(\int_xf\dd x\Big).
\end{align*}
By using the nonnegativity of relative entropy, for almost
every $v$, we have
\begin{align*}
    \int_{\R^d}f(x,v)
    \log\frac{f(x,v)}
    {M(x)\int_{\R^d}f(y,v)\dd y}\dd x\ge 0,
\end{align*}
where, only in this proof, we take $M(x)=\pi^{-d/2}e^{-|x|^2}$, the normalised
Gaussian on $\R^d$. As a consequence, we have
\begin{align*}
&\Big(\int_{\R^d}f(x,v)\dd x\Big)
\Big(\log\int_{\R^d}f(x,v)\dd x\Big)_+\\
\le&{} \int_{\R^d}f(x,v)(\log f(x,v))_+\dd x-\int_{\R^d}f(x,v)\log M(x)\dd x.
\end{align*}
Integrating in $v$, we obtain
\begin{align*}
H\Big(\int_xf\dd x\Big)
&\le H_0+\int_{\R^{2d}}f(x,v)
\Big(|x|^2+\frac d2\log\pi\Big)\dd x\dd v\\
&\lesssim H_0+\|f\|_{L^1_{2,2}},
\end{align*}
which implies the entropy bound in \eqref{H:k-1}.

We also have
$E(\hf)\le\|\kappa\|_{L^\infty}
\|\langle v\rangle^2f\|_{L^1_{x,v}}
\le\|\kappa\|_{L^\infty}E_0$.
Since
$\hf(x,v)\le\|\kappa\|_{L^\infty}
\int_{x_*}f(x_*,v)\dd x_*$,
we have
\begin{align*}
    \sup_x\|\hf(x,\cdot)\|_{L^p_v}
    \le\|\kappa\|_{L^\infty}\|f\|_{L^p_vL^1_x},
    \quad1\le p\le\infty.
\end{align*}

On the other hand, in the case of $\kappa\sim1$
on $\R^d$, or on $\T^d$ with $\inf_x\kappa>0$,
the positive lower bound of $\kappa$ ensures
the mass lower bound
\begin{align*}
    \int_v\hf(x,v)\dd v
    \ge(\inf_x\kappa)\|f\|_{L^1_{x,v}}.
\end{align*}
\end{proof}

\begin{lemma}\label{lem:Des}
Let $A$ and $\kappa$ satisfy Assumption~\ref{ASS:kernel}. Let $f\in L^1_{2,2}(\Do)\cap\logl(\Do;\R_+)$ such that $\|\langle v\rangle^2f\|_{L^1_{x,v}}\le E_0$ and $\|f(\log f)_+\|_{L^1_{x,v}}\le H_0$ for some constants $E_0,\,H_0>0$. Then we have
\begin{equation}
\label{ineq:Des}
\begin{aligned}
&\sum_{i,j=1}^d\int_{\Do}
\langle v\rangle^\gamma f
\Big|\Big[\frac{\nabla_v f}{f},v\Big]_{ij}\Big|^2\dd x\dd v
+\int_{\Do}\langle v\rangle^\gamma f
\Big|\frac{\nabla_v f}{f}\Big|^2\dd x\dd v\\
\le&{}C_\kappa C(f)\big(1+\cD(f)\big).
\end{aligned}
\end{equation}

The constant $C_\kappa>0$ depends only on
$\gamma,d,\kappa,E_0$ and $\|f\|_{L^1_{x,v}}$, together with either
$p>1$ and $\|f\|_{L^p_vL^1_x}$, when this norm is finite, or
$H_0$ and $\|f\|_{L^1_{2,2}}$.

For the constant $C(f)$, we have the following bounds:
\begin{itemize}
\item In the case of $\gamma\in(-\gamma_d,0]$, we have
\begin{equation*}
\label{cst:des:1}
    C(f)\lesssim
    E_0^3.
\end{equation*}

\item Let $p>\frac{d}{d-\gamma}$. In the case of $\gamma\in(0,1]$, we have
\begin{equation*}
\label{cst:des:2}
\begin{aligned}
    C(f)\lesssim{}&
    E_0^2
    \big(
    \|f\|_{L^1_{0,2+\gamma}(\Do)}
    +\|f\|_{L^p_vL^1_x}
    \big).
\end{aligned}
\end{equation*}
\end{itemize}
\end{lemma}

\begin{proof}
The proof follows \cite{Des15,Des16}, with appropriate modifications to account for the spatial kernel $\kappa$. The case $f=0$ is immediate. Otherwise, by normalising $f$ and rescaling $\kappa$, we may assume
$\|f\|_{L^1_{x,v}}=1$ and $\|\kappa\|_{L^\infty}=1$,
with the resulting factors absorbed into $C_\kappa$.
We first perform the calculations where the quotients are defined;
their justification for general densities is given at the end of the proof.

We define $\d_j\defeq\d_{v_j}$ and, for $i,j=1,\ldots,d$, we write
\begin{equation}
\label{q:ij}
\begin{aligned}
&q_{ij}=q_{ij}^{f}(x,v,y,w)=\Big[v-w,\frac{\nabla_v f}{f}(x,v)-\frac{\nabla_v f}{f}(y,w)\Big]_{ij}\\
&=(v_i-w_i)\Big(\frac{\d_jf}{f}(x,v)-\frac{\d_jf}{f}(y,w)\Big)
 -(v_j-w_j)\Big(\frac{\d_if}{f}(x,v)-\frac{\d_if}{f}(y,w)\Big)\\
&=\Big(v_i\frac{\d_jf}{f}(x,v)-v_j\frac{\d_if}{f}(x,v)\Big)
 +\Big(w_j\frac{\d_if}{f}(x,v)-w_i\frac{\d_jf}{f}(x,v)\Big)
 +\Big(v_j\frac{\d_if}{f}(y,w)\\
 &\quad-v_i\frac{\d_jf}{f}(y,w)\Big)+\Big(w_i\frac{\d_jf}{f}(y,w)-w_j\frac{\d_if}{f}(y,w)\Big).
\end{aligned}
\end{equation}
We write $q^f=(q_{ij})_{i,j}$. The entropy dissipation \eqref{dissipation:cross} can be written as
\begin{equation}
\label{D-qf}
\cD(f)=\frac14\int_{\G}\kappa(x-y)|v-w|^\gamma f(x,v)f(y,w)|q^f|^2\dd\eta.
\end{equation}
We have the following identity:
\begin{equation}
\label{eq:q_ij}
\begin{aligned}
&\Big[v,\frac{\nabla_v f}{f}(x,v)\Big]_{ij}
 -w_i\frac{\d_jf}{f}(x,v)+w_j\frac{\d_if}{f}(x,v)\\
=&{}v_i\frac{\d_jf}{f}(y,w)-v_j\frac{\d_if}{f}(y,w)
 -\Big(w_i\frac{\d_j f}{f}(y,w)-w_j\frac{\d_i f}{f}(y,w)\Big)+q_{ij}.
\end{aligned}
\end{equation}

To obtain a better bound, we define the weight function
$\phi(r)=(1+2r)^{-\frac{|\gamma|}{4}-\frac12}$ and we write
$\phi=\phi(|w|^2/2)=\langle w\rangle^{-\frac{|\gamma|}{2}-1}$.
Notice that $|\phi|\le1$, $|\phi'(r)|\le1+\frac{|\gamma|}{2}$ for $r\ge0$,
and $\nabla_w\phi=w\phi'(|w|^2/2)$.

We multiply \eqref{q:ij} by $\phi\kappa(x-y)f(y,w)$ and integrate over
$(y,w)\in\Do$, noting that this annihilates the last bracket term
in \eqref{q:ij}, to derive
\begin{align}
&\int_{\Do}\phi\kappa q_{ij}f(y,w)\dd y\dd w\notag\\
&=\Big(v_i\frac{\d_jf}{f}(x,v)-v_j\frac{\d_if}{f}(x,v)\Big)
 \int_{\Do}\phi\kappa f(y,w)\dd y\dd w\notag\\
&+\frac{\d_if}{f}(x,v)\int_{\Do}\phi\kappa w_jf(y,w)\dd y\dd w
 -\frac{\d_jf}{f}(x,v)\int_{\Do}\phi\kappa w_if(y,w)\dd y\dd w\notag\\
&+v_i\int_{\Do}\phi'\kappa w_jf(y,w)\dd y\dd w
 -v_j\int_{\Do}\phi'\kappa w_if(y,w)\dd y\dd w.
\label{eq: sys1}
\end{align}

In the following calculations, let $i\ne j$. Next we multiply \eqref{q:ij} by $\phi\kappa w_if(y,w)$ and integrate over $(y,w)\in\Do$ to derive
\begin{align}
&\int_{\Do}\phi\kappa q_{ij}w_if(y,w)\dd y\dd w\notag\\
&=\Big(v_i\frac{\d_jf}{f}(x,v)-v_j\frac{\d_if}{f}(x,v)\Big)
  \int_{\Do}\phi\kappa w_if(y,w)\dd y\dd w\notag\\
&+\frac{\d_if}{f}(x,v)\int_{\Do}\phi\kappa w_iw_jf(y,w)\dd y\dd w
 -\frac{\d_jf}{f}(x,v)\int_{\Do}\phi\kappa w_i^2f(y,w)\dd y\dd w\notag\\
&+v_i\int_{\Do}\phi'\kappa w_iw_jf(y,w)\dd y\dd w\notag\\
&-v_j\int_{\Do}(\phi+w_i^2\phi')\kappa f(y,w)\dd y\dd w
 +\int_{\Do}\phi\kappa w_jf(y,w)\dd y\dd w.
\label{eq: sys2}
\end{align}
To obtain the last term, we have used the following calculations
\begin{align*}
&\int_{\Do}\Big(w_i\frac{\d_jf}{f}(y,w)-w_j\frac{\d_if}{f}(y,w)\Big)
 w_i\phi\kappa f(y,w)\dd y\dd w\\
&\qquad=\int_{\Do}w_i^2\phi\kappa\partial_jf(y,w)\dd y\dd w
 -\int_{\Do}w_iw_j\phi\kappa\partial_if(y,w)\dd y\dd w\\
&\qquad=-\int_{\Do}w_i^2w_j\phi'\kappa f(y,w)\dd y\dd w
 +\int_{\Do}w_j\partial_i(w_i\phi)\kappa f(y,w)\dd y\dd w\\
&\qquad=-\int_{\Do}w_i^2w_j\phi'\kappa f(y,w)\dd y\dd w
 +\int_{\Do}w_j(\phi+w_i^2\phi')\kappa f(y,w)\dd y\dd w\\
&\qquad=\int_{\Do}w_j\phi\kappa f(y,w)\dd y\dd w.
\end{align*}
Similarly, testing \eqref{q:ij} by $\phi\kappa w_jf(y,w)$ yields
\begin{align}
&\int_{\Do}\phi\kappa q_{ij}w_jf(y,w)\dd y\dd w\notag\\
&=\Big(v_i\frac{\d_jf}{f}(x,v)-v_j\frac{\d_if}{f}(x,v)\Big)
  \int_{\Do}\phi\kappa w_jf(y,w)\dd y\dd w\notag\\
&+\frac{\d_if}{f}(x,v)\int_{\Do}\phi\kappa w_j^2f(y,w)\dd y\dd w
 -\frac{\d_jf}{f}(x,v)\int_{\Do}\phi\kappa w_iw_jf(y,w)\dd y\dd w\notag\\
&-v_j\int_{\Do}\phi'\kappa w_iw_jf(y,w)\dd y\dd w\notag\\
&+v_i\int_{\Do}(\phi+w_j^2\phi')\kappa f(y,w)\dd y\dd w
 -\int_{\Do}\phi\kappa w_if(y,w)\dd y\dd w.
\label{eq: sys3}
\end{align}
From \eqref{eq: sys1}, \eqref{eq: sys2} and \eqref{eq: sys3}, we obtain a $3\times3$ system of linear equations for $v_i\frac{\d_jf}{f}(x,v)-v_j\frac{\d_if}{f}(x,v)$, $\frac{\d_if}{f}(x,v)$ and $-\frac{\d_jf}{f}(x,v)$. Solving this system by Cramer's formula, we get
\begin{equation*}
    \frac{\d_if}{f}(x,v)=\frac{\Delta_i(f)}{\Delta_\phi(f)}
    \quad\text{and}\quad
    v_i\frac{\d_jf}{f}(x,v)-v_j\frac{\d_if}{f}(x,v)
    =\frac{\Delta_{ij}(f)}{\Delta_\phi(f)}.
\end{equation*}
In the above equation, $\Delta_i(f)$ and $\Delta_{ij}(f)$ are given by
\begin{align*}
\Delta_i(f)&=\operatorname{Det}\Big(
\int_{\Do}\phi\kappa f(y,w)
\begin{pmatrix}
1&q_{ij}+P^{ij}_1&w_i\\
w_i&w_iq_{ij}+P^{ij}_2&w_i^2\\
w_j&w_jq_{ij}+P^{ij}_3&w_iw_j
\end{pmatrix}\dd y\dd w\Big),\\
\Delta_{ij}(f)&=\operatorname{Det}\Big(
\int_{\Do}\phi\kappa f(y,w)
\begin{pmatrix}
q_{ij}+P^{ij}_1&w_j&w_i\\
w_iq_{ij}+P^{ij}_2&w_iw_j&w_i^2\\
w_jq_{ij}+P^{ij}_3&w_j^2&w_iw_j
\end{pmatrix}\dd y\dd w\Big),
\end{align*}
where
\begin{align*}
P^{ij}_1&=(v_jw_i-v_iw_j)\frac{\phi'}{\phi},\\
P^{ij}_2&=v_j-w_j+(v_jw_i^2-v_iw_iw_j)\frac{\phi'}{\phi},\\
P^{ij}_3&=w_i-v_i+(v_jw_iw_j-v_iw_j^2)\frac{\phi'}{\phi}.
\end{align*}
The determinant $\Delta_\phi(f)$ is given by
\begin{align*}
\Delta_\phi(f)&=\operatorname{Det}\Big(
\int_{\Do}\phi\kappa f(y,w)
\begin{pmatrix}
1&w_j&w_i\\
w_i&w_iw_j&w_i^2\\
w_j&w_j^2&w_iw_j
\end{pmatrix}\dd y\dd w\Big).
\end{align*}
We claim, and show below, that
\begin{equation}
\label{low:Delta-f}
    |\Delta_\phi(f)(x)|\ge C_\kappa>0,
    \quad x\in\Omega.
\end{equation}
The dependence of $C_\kappa$ will be specified below.

We note that
\begin{align*}
    |\Delta_i(f)|+|\Delta_{ij}(f)|
    \lesssim_{d,\gamma,\kappa}
    E_0\Big(
        \langle v\rangle
        +\int_{\Do}\kappa f(y,w)
        \langle w\rangle^{-\frac{|\gamma|}{2}}|q^f|
        \dd y\dd w
    \Big).
\end{align*}
Then we have, by \eqref{low:Delta-f} and the Cauchy--Schwarz inequality,
\begin{align*}
    &\int_{\Do}
        \langle v\rangle^\gamma f(x,v)
        \Big|\frac{\nabla_v f}{f}(x,v)\Big|^2\dd x\dd v\\
    \lesssim&{}_{d,\gamma,\kappa}
    C_\kappa^{-2}E_0^2
    \int_{\Do}
        \langle v\rangle^\gamma f(x,v)
        \Big(\langle v\rangle+
        \int_{\Do}\kappa\langle w\rangle^{-\frac{|\gamma|}{2}}
        f(y,w)|q^f|\dd y\dd w\Big)^2
        \dd x\dd v\\
    \lesssim&{}_{d,\gamma,\kappa}
    C_\kappa^{-2}E_0^2
    \int_{\Do}\langle v\rangle^\gamma f(x,v)
    \Big(
        \langle v\rangle^2\\
    &+
        \Big(\int_{\Do}\kappa f(y,w)|q^f|^2|v-w|^\gamma\dd y\dd w\Big)
        \Big(\int_{\Do}\kappa f(y,w)|v-w|^{-\gamma}
        \langle w\rangle^{-|\gamma|}\dd y\dd w\Big)
    \Big)\dd x\dd v\\
    \lesssim&{}_{d,\gamma,\kappa}
    C_\kappa^{-2}E_0^2\Big(
        \|f\|_{L^1_{0,2+\gamma_+}(\Do)}\\
    &\qquad+
        \sup_{x,v}\Big(
        \langle v\rangle^\gamma
        \int_{\Do}f(y,w)\langle w\rangle^{-|\gamma|}
        |v-w|^{-\gamma}\dd y\dd w
        \Big)
        \int_{\G}\kappa f(x,v)f(y,w)|q^f|^2
        |v-w|^\gamma\dd\eta
    \Big)\\
    \lesssim&{}_{d,\gamma,\kappa}
    C_\kappa^{-2}E_0^2\Big(
        \|f\|_{L^1_{0,2+\gamma_+}(\Do)}
        +\cD(f)C^\pm_{-\gamma,-|\gamma|}(f)
    \Big),
\end{align*}
where $C^\pm_{-\gamma,-|\gamma|}$ are given in
\eqref{cst:C-} and \eqref{cst:C+} below.
The same bounds hold for the cross-Fisher information by using
the formula for $\Delta_{ij}(f)$.
Lemma~\ref{bdd:conv}, with $\alpha=-\gamma$ and
$\delta=-|\gamma|$, gives the stated bounds for $C(f)$.

We now give a lower bound for $|\Delta_\phi(f)|$ as in
\eqref{low:Delta-f}. By Lemma~\ref{lem:hydro}, we have
$E(\hf)\le E_0$ and $m(\hf)\ge m_0>0$ uniformly in $x$.

Following the argument of \cite[Lemma 2]{Des15}, and using the
Gram determinant formula and
$\phi\ge(1+L^2)^{-\frac{|\gamma|}{4}-\frac12}$ on $B_w^L$, we have
\begin{align*}
    |\Delta_\phi(f)|
    \gtrsim_d&{}\delta^4(1+L^2)^{-\frac{3|\gamma|}{4}-\frac32}
    \Big(\int_{B_w^L}\hf(x,w)\dd w\Big)\\
    &\times\Big(
    \int_{B_w^L}\hf(x,w)\dd w
    -\sup_{\Omega_w^\delta}
        \int_{\Omega_w^\delta\cap B_w^L}\hf(x,w)\dd w
    \Big)^2,
\end{align*}
where the supremum is taken over all affine slabs
$\Omega_w^\delta$ of width $2\delta$, for which
$|\Omega_w^\delta\cap B_w^L|\le2\delta(2L)^{d-1}$.

We take $L>1$ large enough such that
$E_0/L^2\le m_0/4$.
Then we have
\begin{align*}
    \int_{B_w^L}\hf(x,w)\dd w
    \ge m(\hf)-\frac{E_0}{L^2}
    \ge\frac{3m_0}{4}.
\end{align*}
We split the integration into $\hf\ge K$ and $0\le\hf<K$.
Then we have
\begin{align*}
    \int_{\Omega_w^\delta\cap B_w^L}\hf(x,w)\dd w
    \le2\delta(2L)^{d-1}K
    +\int_{\{\hf(x,w)\ge K\}}\hf(x,w)\dd w.
\end{align*}

If $\|f\|_{L^p_vL^1_x}<+\infty$ for some $p\in(1,\infty)$,
Lemma~\ref{lem:hydro} gives
\begin{align*}
    \int_{\{\hf(x,w)\ge K\}}\hf(x,w)\dd w
    \le\frac{\|\hf(x,\cdot)\|_{L^p_v}^p}{K^{p-1}}
    \le\frac{\|f\|_{L^p_vL^1_x}^p}{K^{p-1}}.
\end{align*}
Taking
$K=\Big(1+\frac{4\|f\|_{L^p_vL^1_x}^p}{m_0}\Big)^{1/(p-1)}$
and $\delta=m_0K^{-1}2^{-(d+2)}L^{-(d-1)}$, we obtain
\begin{align*}
    |\Delta_\phi(f)|
    &\gtrsim_d m_0^3\delta^4
        (1+L^2)^{-\frac{3|\gamma|}{4}-\frac32}\\
    &\gtrsim_d m_0^7
        \Big(1+\frac{4\|f\|_{L^p_vL^1_x}^p}{m_0}\Big)^{-\frac4{p-1}}
        L^{-4(d-1)}
        (1+L^2)^{-\frac{3|\gamma|}{4}-\frac32}>0.
\end{align*}

Alternatively, Lemma~\ref{lem:hydro} gives
$H(\hf)\le C(H_0+\|f\|_{L^1_{2,2}})$ uniformly in $x$.
Hence, for $K>1$, we have
\begin{align*}
    \int_{\{\hf(x,w)\ge K\}}\hf(x,w)\dd w
    \le\frac{C(H_0+\|f\|_{L^1_{2,2}})}{\log K}.
\end{align*}
Taking
$K=e^{4C(H_0+\|f\|_{L^1_{2,2}})/m_0+1}$
and $\delta=m_0K^{-1}2^{-(d+2)}L^{-(d-1)}$, we have
\begin{align*}
    \sup_{\Omega_w^\delta}
    \int_{\Omega_w^\delta\cap B_w^L}\hf(x,w)\dd w
    \le\frac{m_0}{2}.
\end{align*}
Substituting these estimates into the determinant bound, we obtain
\begin{align*}
    |\Delta_\phi(f)|
    \gtrsim_d m_0^3\delta^4
    (1+L^2)^{-\frac{3|\gamma|}{4}-\frac32}>0.
\end{align*}
Thus we obtain \eqref{low:Delta-f}, with the stated dependence
of the constants.

\end{proof}

\begin{corollary}\label{coro:Des}
Let $A$ be given by \eqref{A-1}. Let $f$ satisfy the
assumptions on $f$ in Lemma~\ref{lem:Des}. Let
$$
\kappa(x)\sim\exp(-k_0\langle x\rangle),
\quad x\in\R^d,\quad k_0>0.
$$
Then we have
$H(\hf)\lesssim_\kappa H_0+\|f\|_{L^1_{x,v}}$ and
$\|\hf(x,\cdot)\|_{L^p_v}
\le\|\kappa\|_{L^\infty}\|f\|_{L^p_vL^1_x}$
for any $p\in(1,\infty)$, together with
\begin{equation}
    E(\hf)\le\|\kappa\|_{L^\infty}E_0,\quad
    m(\hf)\ge C(f,\kappa)\exp(-k_0\langle x\rangle)
    \|f\|_{L^1_{x,v}},
\end{equation}
where $C(f,\kappa)>0$ is independent of $x$.

For any $R>0$, the weighted Fisher information bound
\eqref{ineq:Des} holds on $B_x^R\times\R_v^d$:
\begin{equation}
\begin{aligned}
&\sum_{i,j=1}^d\int_{B_x^R\times\R_v^d}
    \langle v\rangle^\gamma f
    \Big|\Big[\frac{\nabla_v f}{f},v\Big]_{ij}\Big|^2\dd x\dd v
 +\int_{B_x^R\times\R_v^d}
    \langle v\rangle^\gamma f
    \Big|\frac{\nabla_v f}{f}\Big|^2\dd x\dd v\\
\le&{}C_\kappa(R) C(f)\big(1+\cD(f)\big).
\end{aligned}
\end{equation}
Here $C(f)$ has the same bounds as in Lemma~\ref{lem:Des}.
The constant $C_\kappa(R)>0$ depends on
$R,\gamma,d,\kappa,E_0,\|f\|_{L^1_{x,v}}$ and $C(f,\kappa)$,
together with either $H_0$ or $p>1$ and
$\|f\|_{L^p_vL^1_x}$, when this norm is finite.
\end{corollary}

\begin{proof}
The entropy bound follows from Jensen's inequality after
normalising $\kappa$ by its $L^1$-norm.
The energy and mixed-norm bounds follow as in
Lemma~\ref{lem:hydro}. Using
$\langle x-x_*\rangle\le\langle x\rangle+|x_*|$, we have
\begin{align*}
    \int_v\hf(x,v)\dd v
    &\gtrsim_\kappa
    \exp(-k_0\langle x\rangle)
    \int_{\Do}e^{-k_0|x_*|}f(x_*,v)\dd x_*\dd v\\
    &\gtrsim_{f,\kappa}
    \exp(-k_0\langle x\rangle)\|f\|_{L^1_{x,v}}.
\end{align*}
This gives a positive uniform mass lower bound on $B_x^R$.
The local Fisher-information estimate then follows from the
proof of Lemma~\ref{lem:Des}.
\end{proof}

The following lemma gives bounds on the convolution
$\|f*_v|v|^\alpha\|_{L^1_x}$, which are used in
Lemma~\ref{lem:Des}.

\begin{lemma}
\label{bdd:conv}
Let $\alpha\in(-d,\infty)$ and $\delta\in\R$.
Let $f\in L^1_{2,2}\cap\logl(\Do;\R_+)$.
The following estimates hold:
\begin{itemize}
\item
If $\alpha\in[0,\infty)$, then
\begin{equation*}
    \|(\langle v\rangle^\delta f)*_v|v|^\alpha\|_{L^1_x}
    \le C^+_{\alpha,\delta}(f)\langle v\rangle^\alpha
    \quad\forall v\in\R^d,
\end{equation*}
where
\begin{equation}
\label{cst:C+}
    C^+_{\alpha,\delta}(f)
    \lesssim\|\langle v\rangle^{\alpha+\delta}f\|_{L^1(\Do)}.
\end{equation}

\item
If $\alpha\in(-d,0)$, then we have
\begin{equation*}
    \|(\langle v\rangle^\delta f)*_v|v|^\alpha\|_{L^1_x}
    \le C^-_{\alpha,\delta}(f)\langle v\rangle^\alpha
    \quad\forall v\in\R^d,
\end{equation*}
where
\begin{equation}
\label{cst:C-}
    C^-_{\alpha,\delta}(f)
    \lesssim
    \|\langle v\rangle^{|\alpha|+\delta}f\|_{L^1(\Do)}
    +\|\langle v\rangle^{|\alpha|+\delta}f\|_{L^p_vL^1_x}
\end{equation}
for some $p>\frac{d}{d-|\alpha|}$.
\end{itemize}
\end{lemma}

\begin{proof}
The proof follows \cite{carrillo2024landau}, with an appropriate
modification to account for the spatial variable $x$.

For the case of $\alpha\in[0,+\infty)$, we apply the inequality
$|v-v_*|^\alpha\lesssim\langle v\rangle^\alpha+
\langle v_*\rangle^\alpha$ and use
$\langle v_*\rangle^\delta
\le\langle v_*\rangle^{\alpha+\delta}$.

We then deal with the case of $\alpha\in(-d,0)$.
We consider the cases $|v|\le1$ and $|v|\ge1$ separately.
In the case of $|v|\le1$, we only need to show the boundedness of
$\|(\langle v\rangle^\delta f)*_v|v|^\alpha\|_{L^1_x}$.
Notice that
\begin{equation}
\label{split:bdd}
\begin{aligned}
    &\int_{\Do}f(x_*,v_*)|v-v_*|^\alpha
    \langle v_*\rangle^\delta\dd x_*\dd v_*\\
    =&{}\Big(
    \int_{\Do\cap\{|v-v_*|\ge1\}}
    +\int_{\Do\cap\{|v-v_*|\le1\}}\Big)
    f(x_*,v_*)|v-v_*|^\alpha
    \langle v_*\rangle^\delta\dd x_*\dd v_*\\
    \le&{}\|\langle v\rangle^\delta f\|_{L^1(\Do)}
    +\int_{\Do\cap\{|v-v_*|\le1\}}
    f(x_*,v_*)|v-v_*|^\alpha
    \langle v_*\rangle^\delta\dd x_*\dd v_*,
\end{aligned}
\end{equation}
where we have used the inequality $|v-v_*|^\alpha\le1$
to estimate the first term.
On the region $\Do\cap\{|v-v_*|\le1\}$, we have
\begin{equation}
\label{LpL1}
\begin{aligned}
    &\int_{\Do\cap\{|v-v_*|\le1\}}
    f(x_*,v_*)|v-v_*|^\alpha
    \langle v_*\rangle^\delta\dd x_*\dd v_*\\
    =&{}\int_{\Do}f(x_*,v_*)\langle v_*\rangle^\delta
    \chi_{B_1}(v-v_*)|v-v_*|^\alpha\dd x_*\dd v_*\\
    =&{}\Big[
    \big(\int_{x_*}\langle\cdot\rangle^\delta
    f(x_*,\cdot)\dd x_*\big)
    *_v(\chi_{B_1}|\cdot|^\alpha)\Big](v)\\
    \le&{}\|\langle v\rangle^\delta f\|_{L^p_vL^1_x}
    \||v|^\alpha\|_{L^{p'}(B_1)},
\end{aligned}
\end{equation}
where $B_1=\{v\mid |v|\le1\}$ and $\chi_{B_1}$
denotes its characteristic function, and
$\frac1p+\frac1{p'}=1$.
Notice that $p>\frac{d}{d-|\alpha|}$ ensures that
$|\alpha|p'<d$, and hence
$|v|^\alpha\in L^{p'}(B_1)$.
If we first apply Young's inequality and then integrate over
$x_*$, we obtain the bound in terms of
$\|\langle v\rangle^\delta f\|_{L^1_xL^p_v}$.
This argument also applies to the following estimates.

In the case of $|v|\ge1$, we split the integration region as follows
\begin{equation*}
\begin{aligned}
    &\int_{\Do}f(x_*,v_*)|v-v_*|^\alpha
    \langle v_*\rangle^\delta\dd x_*\dd v_*\\
    =&{}\Big(
    \int_{\Do\cap\{|v_*|\le\frac{|v|}{2}\}}
    +\int_{\Do\cap\{|v_*|\ge\frac{|v|}{2}\}}\Big)
    f(x_*,v_*)|v-v_*|^\alpha
    \langle v_*\rangle^\delta\dd x_*\dd v_*\\
    \le&{}\Big(\frac{|v|}{2}\Big)^\alpha
    \int_{\Do\cap\{|v_*|\le\frac{|v|}{2}\}}
    f(x_*,v_*)\langle v_*\rangle^\delta\dd x_*\dd v_*\\
    &+\Big(\frac{|v|}{2}\Big)^\alpha
    \int_{\Do\cap\{|v_*|\ge\frac{|v|}{2}\}}
    f(x_*,v_*)\Big(\frac{|v-v_*|}{|v_*|}\Big)^\alpha
    \langle v_*\rangle^\delta\dd x_*\dd v_*\\
    \le&{}2^{-\alpha}|v|^\alpha
    \|\langle v\rangle^\delta f\|_{L^1(\Do)}\\
    &+2^{-\alpha}|v|^\alpha
    \int_{\Do\cap\{|v_*|\ge\frac{|v|}{2}\}}
    f(x_*,v_*)\Big(\frac{|v-v_*|}{|v_*|}\Big)^\alpha
    \langle v_*\rangle^\delta\dd x_*\dd v_*.
\end{aligned}
\end{equation*}
Similarly to \eqref{LpL1}, on the region
$\Do\cap\{|v_*|\ge\frac{|v|}{2}\}$, we have
\begin{align*}
    &\int_{\Do\cap\{|v_*|\ge\frac{|v|}{2}\}}
    f(x_*,v_*)\Big(\frac{|v-v_*|}{|v_*|}\Big)^\alpha
    \langle v_*\rangle^\delta\dd x_*\dd v_*\\
    \le&{}\Big(
    \int_{\Do\cap\{|v-v_*|\ge1\}}
    +\int_{\Do\cap\{|v-v_*|\le1\}}\Big)
    f(x_*,v_*)|v-v_*|^\alpha
    \langle v_*\rangle^{\delta+|\alpha|}\dd x_*\dd v_*\\
    \le&{}\|\langle v\rangle^{\delta+|\alpha|}f\|_{L^1(\Do)}
    +\Big[
    \big(\int_{x_*}f(x_*,\cdot)
    \langle\cdot\rangle^{\delta+|\alpha|}\dd x_*\big)
    *_v(\chi_{B_1}|\cdot|^\alpha)\Big](v)\\
    \le&{}\|\langle v\rangle^{\delta+|\alpha|}f\|_{L^1(\Do)}
    +\|\langle v\rangle^{\delta+|\alpha|}f\|_{L^p_vL^1_x}
    \||v|^\alpha\|_{L^{p'}(B_1)}.
\end{align*}

Since $|v|^\alpha\lesssim\langle v\rangle^\alpha$
for $|v|\ge1$, we combine all the cases and conclude that
\begin{align*}
    &\int_{\Do}f(x_*,v_*)|v-v_*|^\alpha
    \langle v_*\rangle^\delta\dd x_*\dd v_*\\
    \lesssim&{}\langle v\rangle^\alpha
    \big(
    \|\langle v\rangle^{|\alpha|+\delta}f\|_{L^1(\Do)}
    +\|\langle v\rangle^{|\alpha|+\delta}f\|_{L^p_vL^1_x}
    \big).
\end{align*}
\end{proof}

\begin{corollary}
\label{coro:conv}
Let $\alpha\in(-d,0)$ and $\delta\in\R$.
Let $f\in L^1_{2,2}(\Do)\cap\logl(\Do;\R_+)$.
Then we have
\begin{equation*}
    \sup_{v\in\R^d}
    \|(\langle v\rangle^\delta f)*_v|v|^\alpha\|_{L^1_x}
    \lesssim
    \|\langle v\rangle^\delta f\|_{L^1(\Do)}
    +
    \|\langle v\rangle^\delta f\|_{L^p_vL^1_x}
\end{equation*}
for some $p>\frac{d}{d-|\alpha|}$.
\end{corollary}

\begin{proof}
We repeat \eqref{split:bdd}--\eqref{LpL1} in the proof of
Lemma~\ref{bdd:conv}, which hold for all $v\in\R^d$.
\end{proof}

\section{Variational characterisation}\label{sec:vc}
In this section, we will first introduce and discuss the transport grazing rate equation in Section~\ref{TGE}. We will prove a chain rule for transport grazing rate equations in Section~\ref{sub:chain-rule}. We will prove the variational characterisation in Theorem~\ref{thm:main} for the fuzzy Landau equation \eqref{Landau-fuz} in Section~\ref{VC}.

For simplicity, we write the domain as $\Do$, which denotes either $\R^{2d}$ or $\Dot$, unless otherwise stated.

\subsection{Transport grazing rate equation}\label{TGE}

We consider curves $(\mu_t)_{t\in[0,T]}\subset\cP(\Do)$ driven by transport and grazing mechanisms through a transport grazing rate continuity equation
\begin{equation}
\label{TGRE}
    \d_t\mu_t+v\cdot\nabla_x\mu_t+\frac12\widetilde\nabla\cdot\cU_t=0,
\end{equation}
where $\cU_t\in\cM(\G;\R^d)$ denotes the grazing rate, and $\cM(\G;\R^d)$ denotes the space of vector-valued signed Borel measures with bounded total variation on $\G$. The space $\cM(\G;\R^d)$ is endowed with the weak* topology induced by $C_0(\G;\R^d)$, the space of continuous functions vanishing at infinity. We recall that the gradient operator $\widetilde\nabla$ and the associated divergence operator $\widetilde\nabla\cdot$ are defined, respectively, in \eqref{gradient} and \eqref{tilde:IP}. Also, we denote by
\begin{align*}
    \dd\eta=\dd x\dd v\dd x_*\dd v_*
\end{align*}
the product measure on $\G$.

\begin{definition}
We say that $(\mu_t,\cU_t)$ is a solution pair to the transport grazing rate equation \eqref{TGRE} if
\begin{enumerate}
    \item $t\mapsto\mu_t$ is a weakly continuous map from $[0,T]$ to $\cP(\Do)$;
    \item $(\cU_t)_{t\in[0,T]}$ is a weak* measurable family in $\cM(\G;\R^d)$;
    \item $\int_0^T|\cU_t|(\G)\dd t<\infty$;
    \item For any $\varphi\in C^\infty_c(\Do)$, the following equality holds in the sense of distributions on $(0,T)$, with all terms locally integrable in time:
    \begin{equation*}
        \frac{d}{dt}\int_{\Do}\varphi\dd\mu_t
        -\int_{\Do}v\cdot\nabla_x\varphi\dd\mu_t
        =\frac12\int_{\G}\widetilde\nabla\varphi\cdot\dd\cU_t.
    \end{equation*}
\end{enumerate}
We denote the set of such pairs $(\mu_t,\cU_t)$ by $\TGRE_T$.
\end{definition}

If the pair $(\mu_t,\cU_t)$ has densities $f_t$ and $U_t$ with respect to the product measures on $\Do$ and $\G$, respectively, we will also write $(f_t,U_t)\in\TGRE_T$.

We define the function $\alpha:\R_+\times\R^d\to[0,+\infty]$ by letting
\begin{equation}
\label{def:alpha}
\alpha(s,u)=\left\{
\begin{aligned}
&\frac{|u|^2}{2s},\quad s>0,\\
&0,\quad u=0,\,s=0,\\
&+\infty,\quad u\neq0,\,s=0.
\end{aligned}
\right.
\end{equation}
Notice that the function $(a,b)\mapsto |a|^2/b$ is jointly convex on $\R^d\times(0,\infty)$. The function $\alpha$ is lower semicontinuous, convex and 1-homogeneous, i.e. $\alpha(rs,ru)=r\alpha(s,u)$ for all $r>0$.

Next we define the action functional for $\mu\in\cP(\Do)$ and $\cU\in\cM(\G;\R^d)$.
\begin{definition}[Action functional]
\label{def: action}
We define the action functional
\begin{equation}
\label{eq:action}
    \cA(\mu,\cU)
    =\int_{\G}
    \alpha\Big(\frac{\dd\mu_1}{\dd\lambda},
               \frac{\dd\cU}{\dd\lambda}\Big)\dd\lambda,
\end{equation}
where we define
\begin{align*}
    \mu_1(\dd\eta)
    :=\kappa(x-x_*)\mu(\dd x\dd v)\mu(\dd x_*\dd v_*),
    \qquad
    \lambda:=\mu_1+|\cU|.
\end{align*}
\end{definition}
Notice that $\mu_1$ is a finite nonnegative Borel measure on $\G$, and $\cA(\mu,\cU)$ is independent of the choice of dominating measure since $\alpha$ is 1-homogeneous.

The following lemma shows that if $\mu$ has a density with respect to Lebesgue measure $\cL$, then finiteness of the action functional ensures the existence of a density for $\cU$.
\begin{lemma}
\label{lem:U:density}
If $\mu=f\cL\in\cP(\Do)$ and $\cU\in\cM(\G;\R^d)$ satisfy $\cA(\mu,\cU)<+\infty$, then there exists a Borel function $M:\G\to\R^d$ such that
$\cU=Mff_*\kappa\dd\eta=U\dd\eta$ and
\begin{align*}
    \cA(\mu,\cU)
    =\frac12\int_{\G}|M|^2ff_*\kappa\dd\eta
    =\frac12\int_{\G}\frac{|U|^2}{ff_*\kappa}\dd\eta,
\end{align*}
with the quotient interpreted according to \eqref{def:alpha}.
\end{lemma}
\begin{proof}
The proof follows \cite{erbar2023gradient} and \cite{carrillo2024landau}. We have
\begin{equation}
\label{bdd:A:ac}
    \cA(\mu,\cU)
    =\int_{\G}
    \alpha\Big(\frac{\dd\mu_1}{\dd\lambda},
               \frac{\dd\cU}{\dd\lambda}\Big)\dd\lambda
    <+\infty.
\end{equation}
Let $S\subset\G$ be a Borel set such that $\mu_1(S)=0$. Then
$\frac{\dd\mu_1}{\dd\lambda}=0$ $\lambda$-almost everywhere on $S$.
By the definition of $\alpha$ in \eqref{def:alpha} and the finiteness in \eqref{bdd:A:ac}, we have
$\frac{\dd\cU}{\dd\lambda}=0$ $\lambda$-almost everywhere on $S$.
Hence, $\cU$ is absolutely continuous with respect to $\mu_1$.
Then there exists a Borel function $M$ such that
$\cU=M\mu_1=Mff_*\kappa\dd\eta$.
Substituting this expression into \eqref{eq:action} gives the claimed formula.
\end{proof}

For the pair $(\mu,\cU)$ with densities $(f,U)$ with respect to the product measures, we also write $\cA(f,U)\defeq\cA(\mu,\cU)$.

\begin{lemma}[Lower semicontinuity]
Let $\mu_n\rightharpoonup\mu$ in $\cP(\Do)$ and
$\cU_n\overset{*}{\rightharpoonup}\cU$ in $\cM(\G;\R^d)$. We have
\begin{align*}
    \cA(\mu,\cU)\le\liminf_{n\to\infty}\cA(\mu_n,\cU_n).
\end{align*}
\end{lemma}
\begin{proof}
Weak convergence of $\mu_n$ implies weak convergence of the corresponding product measures. The result is then a direct consequence of \cite[Theorem 3.4.3]{But89}, applied to the integrand $\kappa^{-1}\alpha$, using the continuity and positivity of $\kappa$.
\end{proof}

We have the following integrability estimates on the grazing rate.
\begin{lemma}[Integrability estimates]\label{lem:bdd:U}
Let $(f_t,U_t)\in\TGRE_T$. If
\begin{align*}
    C_A\defeq\int_0^T\cA(f_t,U_t)\dd t<+\infty
    \quad\text{and}\quad
    C_E=\int_0^T\cE_{2,2+|\gamma|}(f_t)\dd t<+\infty,
\end{align*}
then, for $\gamma\in(-\gamma_d,1]$, we have
\begin{equation}
\label{int:U:gamma:0}
\begin{aligned}
    &\int_0^T\int_{\G}|U_t|
    \big(\langle v\rangle^{1+\frac{|\gamma|}{2}}
        +\langle v_*\rangle^{1+\frac{|\gamma|}{2}}
        +\langle x\rangle+\langle x_*\rangle\big)\dd\eta\dd t\\
    &\qquad\le4\sqrt{\|\kappa\|_{L^\infty}C_AC_E}.
\end{aligned}
\end{equation}
When $\gamma\in(-\gamma_d,-2)$, if in addition
$f\in L^1([0,T];L^p_vL^1_x)$ for some
$p>\frac{d}{d-|2+\gamma|}$, then we have
\begin{equation}
\label{int:U:gamma:v-v*:0}
    \int_0^T\int_{\G}|U_t||v-v_*|^{1+\frac{\gamma}{2}}
    \dd\eta\dd t
    \le\sqrt{CC_A},
\end{equation}
where
\begin{align*}
    C\lesssim_{d,\gamma,p,\kappa}
    T+\|f\|_{L^1([0,T];L^p_vL^1_x)}.
\end{align*}
\end{lemma}
\begin{proof}
We first show \eqref{int:U:gamma:0}.
By the Cauchy--Schwarz inequality, we have
\begin{equation}
\label{U:int-1}
\begin{aligned}
    &\int_0^T\int_{\G}|U_t|
    \big(\langle v\rangle^{1+\frac{|\gamma|}{2}}+\langle x\rangle\big)
    \dd\eta\dd t\\
    &\le\sqrt2
    \Big(\int_0^T\int_{\G}
        \frac{|U_t|^2}{ff_*\kappa}\dd\eta\dd t\Big)^{\frac12}
    \Big(\int_0^T\int_{\G}
        \big(\langle v\rangle^{2+|\gamma|}+\langle x\rangle^2\big)
        ff_*\kappa\dd\eta\dd t\Big)^{\frac12}\\
    &\le2\sqrt{\|\kappa\|_{L^\infty}C_AC_E}.
\end{aligned}
\end{equation}
The same bound holds for the starred terms, which gives \eqref{int:U:gamma:0}.

Concerning \eqref{int:U:gamma:v-v*:0}, similarly to \eqref{U:int-1}, we have
\begin{align*}
    &\int_0^T\int_{\G}|U_t||v-v_*|^{1+\frac{\gamma}{2}}
    \dd\eta\dd t\\
    &\le(2C_A)^{\frac12}
    \Big(\int_0^T\int_{\G}
        \kappa ff_*|v-v_*|^{2+\gamma}\dd\eta\dd t\Big)^{\frac12}.
\end{align*}
By using Corollary~\ref{coro:conv} and $\int_{\Do}f\dd x\dd v=1$, we have
\begin{align*}
    \int_0^T\int_{\G}\kappa ff_*|v-v_*|^{2+\gamma}\dd\eta\dd t
    \lesssim_{d,\gamma,p,\kappa}
    T+\|f\|_{L^1([0,T];L^p_vL^1_x)}.
\end{align*}
This proves \eqref{int:U:gamma:v-v*:0}.
\end{proof}

\subsection{Chain rule}\label{sub:chain-rule}

In this subsection, we prove the following chain rule for the transport grazing rate equation
\begin{equation*}
\d_t\mu_t+v\cdot\nabla_x\mu_t+\frac12\widetilde\nabla\cdot\cU_t=0.
\end{equation*}
\begin{theorem}[Chain rule]
\label{thm:chain-rule}
Let $\gamma\in(-\gamma_d,1]$. Let $(\mu_t,\cU_t)\in\TGRE_T$ such that $|\cH(\mu_0)|<+\infty$ and $(\mu_t)_{t\in[0,T]}\subset\cP_{2,2+|\gamma|}(\Do)$. We assume
\begin{equation}
\label{CR:AD}
\int_0^T\cD(\mu_t)\dd t<+\infty
\quad\text{and}\quad
\int_0^T\cA(\mu_t,\cU_t)\dd t<+\infty.
\end{equation}
Let $\mu_t=f_t\cL$.
\begin{itemize}
    \item In the case of $\gamma\in(-\gamma_d,0)$, we assume
    \begin{align*}
        &\int_0^T\cE_{2,0}(\mu_t)\dd t<+\infty,\quad
        \sup_{t\in[0,T]}\cE_{0,2+|\gamma|}(\mu_t)<+\infty,\\
        &\text{and}\quad
        \langle v\rangle^{2+|\gamma|}f
        \in L^\infty([0,T];L^p_vL^1_x)
    \end{align*}
    for some $p>\frac{d}{d-|\gamma|}$.

    \item In the case of $\gamma=0$, we assume
    $\sup_{t\in[0,T]}\cH(\mu_t)<+\infty$ and
    \begin{equation*}
        \int_0^T\cE_{2,0}(\mu_t)\dd t<+\infty
        \quad\text{and}\quad
        \sup_{t\in[0,T]}\cE_{0,2}(\mu_t)<+\infty.
    \end{equation*}

    \item In the case of $\gamma\in(0,1]$, we assume 
    \begin{align*}
        &\int_0^T\cE_{2,0}(\mu_t)\dd t<+\infty,\quad
        \sup_{t\in[0,T]}\cE_{0,2+\gamma}(\mu_t)<+\infty,\\
        &\text{and}\quad
        f\in L^\infty([0,T];L^p_vL^1_x)
    \end{align*}
    for some $p>\frac{d}{d-\gamma}$.
\end{itemize}

Then $|\cH(\mu_t)|<+\infty$ for all $t\in[0,T]$, and, writing $\mu_t=f_t\cL$, the following chain rule holds:
\begin{equation}
\label{int:chain-rule}
\cH(\mu_t)-\cH(\mu_s)
=\frac12\int_s^t\int_{\G}
\widetilde\nabla\log f_r\cdot\dd\cU_r\dd r
\end{equation}
for all $0\le s\le t\le T$. Moreover, the map $t\mapsto\cH(\mu_t)$ is absolutely continuous and we have
\begin{equation}
\label{ineq:chain-rule}
\frac{d}{dt}\cH(\mu_t)
=\frac12\int_{\G}\widetilde\nabla\log f_t\cdot\dd\cU_t
\end{equation}
for almost every $t\in[0,T]$.
\end{theorem}

\begin{remark}
\begin{itemize}
    \item On the domain $\Dot$, we do not need to assume the moment bounds in $x$, since $\T^d$ is compact.

    \item The conclusion $|\cH(\mu_t)|<+\infty$ implies that $\mu_t=f_t\cL$ for all $t\in[0,T]$. Then \eqref{CR:AD} and Lemma~\ref{lem:U:density} imply that $\cU_t=U_t\dd\eta$ for almost every $t$. Moreover, $U_t=0$ almost everywhere on $\{ff_*=0\}$. By the Cauchy--Schwarz inequality, we have
    \begin{equation*}
        \frac12\int_{\G}
        |\widetilde\nabla\log f_t\cdot U_t|\dd\eta
        \le\sqrt{\cD(f_t)\cA(f_t,U_t)}.
    \end{equation*}
    Hence, the right-hand side of \eqref{ineq:chain-rule} is well-defined for almost every $t$ and is integrable in time.

    \item In the proof, we show \eqref{CT:beta} using the Fisher information decomposition in \eqref{limit:fishers} and Lemma~\ref{lem:Des}. For $\gamma\neq0$, the constants in the weighted Fisher and cross-Fisher information bounds are controlled by the moment bounds and $\|f\|_{L^p_vL^1_x}$, which are uniformly bounded by assumption.

    In the Maxwellian case $\gamma=0$, the assumed entropy bound, together with the moment bounds and the uniform spatial second-moment bound obtained below, ensures the corresponding Fisher information bounds.
\end{itemize}
\end{remark}

\begin{remark}[Assumptions on curves]
\begin{itemize}

    \item There are $L^p_vL^1_x$ assumptions on curves in both the soft and hard potential cases, though they play different roles. In the hard potential case $\gamma\in(0,1]$, the $L^p$-assumptions are used to bound the weighted Fisher information; see Lemma~\ref{lem:Des}. In the soft potential case $\gamma\in(-\gamma_d,0)$, the $L^p$-assumptions are used to control
    $\|(\langle v\rangle^2f)*_v|v|^\gamma\|_{L^\infty_vL^1_x}$;
    see, for example, \eqref{use:soft}.

    \item Theorem~\ref{thm:chain-rule} also holds if we replace $L^p_vL^1_x$ in the assumptions by $L^1_xL^p_v$, by Minkowski's inequality.
\end{itemize}
\end{remark}

We prove Theorem~\ref{thm:chain-rule} in three steps: we regularise curves $(\mu_t,\cU_t)\in\TGRE_T$ in $x$, $v$ and $t$, and discuss the properties of the regularised curves in Section~\ref{subsec:reg}. We derive a regularised chain rule and pass to the limit in Section~\ref{subsec:chain-rule} and Section~\ref{subsec:limit}, respectively.

A similar chain rule for spatially homogeneous grazing rate equations
\begin{equation}
\label{homo:GR}
\d_t\mu_t+\frac12\widetilde\nabla\cdot\cU_t=0,
\quad\mu_t\in\cP(\R^d),
\quad\cU_t\in\cM(\R^{2d};\R^d)
\end{equation}
has been established in \cite{carrillo2024landau}. In the proof of Theorem~\ref{thm:chain-rule}, we follow \cite{carrillo2024landau} to treat the collision terms, with appropriate modifications to account for the spatial variable $x$. We follow \cite{EH25} to treat the transport part $v\cdot\nabla_x\mu_t$ that appears in the transport grazing rate equation \eqref{TGRE}. The second author and Erbar \cite{EH25} consider a chain rule for a transport collision rate equation related to a delocalised fuzzy Boltzmann equation.

\subsubsection{Regularisation}\label{subsec:reg}

\begin{enumerate}
    \item \underline{Regularisation in $x$.}

Let $M(x)=c\exp(-\langle x\rangle)$, where we choose $c>0$ such that $\|M\|_{L^1(\R^d)}=1$.
We define the sequence of mollifiers with parameter $\alpha\in(0,1)$ by
\begin{equation*}
M_\alpha(x)=c\alpha^{-d}\exp\big(-\langle{x}/{\alpha}\rangle\big).
\end{equation*}
On $\T^d$, we use the periodisation of $M_\alpha$ in the spatial convolutions.

We define the regularised pair in the following way:
\begin{equation}
\label{def:reg:alpha}
\mu^\alpha_t=\mu_t*_xM_\alpha,
\quad
\cU^\alpha_t=\cU_t*_{(x,x_*)}M_\alpha.
\end{equation}
Here the convolution in $(x,x_*)$ is taken independently in each variable.
For test functions $\phi:\Do\to\R$ and $\Phi:\G\to\R$, we have
\begin{align*}
\int_{\Do}\phi\dd\mu^\alpha_t
&=\int_{\Do}(\phi*_xM_\alpha)\dd\mu_t,\\
\int_{\G}\Phi\dd\cU^\alpha_t
&=\int_{\G}(\Phi*_{(x,x_*)}M_\alpha)\dd\cU_t.
\end{align*}

    \item \underline{Regularisation in $v$.}

We regularise $(\mu^\alpha_t,\cU^\alpha_t)$ in $v$ and independently in $(v,v_*)$ by convolution with $M_\beta$, $\beta\in(0,1)$, in the following way:
\begin{equation*}
\mu^{\alpha,\beta}_t=\mu^\alpha_t*_vM_\beta,
\quad
\cU^{\alpha,\beta}_t=\cU^\alpha_t*_{(v,v_*)}M_\beta.
\end{equation*}

    \item \underline{Regularisation in $t$.}

Let $\eta\in C^\infty_c(\R;\R_+)$ satisfy
$\supp(\eta)\subset[-1,1]$ and $\int_\R\eta=1$.
For $\delta\in(0,T/2)$, we define
$\eta^\delta(t)=\delta^{-1}\eta(t/\delta)$.
We regularise $(\mu^{\alpha,\beta}_t,\cU^{\alpha,\beta}_t)$ in time by
\begin{equation*}
\mu^{\alpha,\beta,\delta}_t
=\int_{-\delta}^{\delta}\mu^{\alpha,\beta}_{t-s}\eta^\delta_s\dd s,
\quad
\cU^{\alpha,\beta,\delta}_t
=\int_{-\delta}^{\delta}\cU^{\alpha,\beta}_{t-s}\eta^\delta_s\dd s
\end{equation*}
for $t\in I^\delta\defeq[\delta,T-\delta]$.

    \item \underline{A lower bound.}

We define
\begin{equation}
\label{def:g}
g(t,x,v)=Z\big(\langle v\rangle+\langle x-tv\rangle\big)^{-a}
\end{equation}
on $\R^{2d}$, and we take $g(t,x,v)=Z\langle v\rangle^{-a}$ on $\Dot$, where we choose $a>2d+2+|\gamma|$ and $Z>0$ such that $\|g_t\|_{L^1(\Do)}=1$.

Notice that $g$ solves the transport equation
\begin{equation*}
(\d_t+v\cdot\nabla_x)g=0.
\end{equation*}
Moreover, we have
\begin{equation*}
g(t,x,v)\ge C(T)\big(\langle v\rangle+\langle x\rangle\big)^{-a},
\qquad
|\d_tg|+|\nabla_vg|\lesssim_Tg.
\end{equation*}

For $\theta\in(0,1)$, we regularise
$(\mu^{\alpha,\beta,\delta}_t,\cU^{\alpha,\beta,\delta}_t)$ by
\begin{equation}
\label{def:reg:theta}
\mu^{\alpha,\beta,\delta,\theta}_t
=(1-\theta)\mu^{\alpha,\beta,\delta}_t+\theta g_t\cL,
\quad
\cU^{\alpha,\beta,\delta,\theta}_t
=(1-\theta)\cU^{\alpha,\beta,\delta}_t.
\end{equation}
We denote by $f^{\alpha,\beta,\delta,\theta}_t$ and
$U^{\alpha,\beta,\delta,\theta}_t$ the densities with respect to Lebesgue measure on $\Do$ and $\G$, respectively.
\end{enumerate}

A vanishing regularisation parameter means that the corresponding regularisation is omitted; in particular, $I^0=[0,T]$.

\begin{remark}
We use the lower bound regularisation \eqref{def:g} to obtain a positive lower bound for the regularised density that is uniform in time. The resulting logarithmic bounds are needed, for example, in Lemma~\ref{lem:reg-curve}-3.

On $\R^{2d}$, by testing the transport grazing rate equation with truncated spatial second moments, we also have
\begin{equation*}
\sup_{t\in[0,T]}\cE_{2,0}(\mu_t)
\lesssim
\cE_{2,0}(\mu_0)
+T^2\sup_{t\in[0,T]}\cE_{0,2}(\mu_t)
<+\infty.
\end{equation*}
\end{remark}

We have the following properties of the regularised curves
$(f^{\alpha,\beta,\delta,\theta}_t,U^{\alpha,\beta,\delta,\theta}_t)$.

\begin{lemma}\label{lem:reg-curve}
\begin{enumerate}[1.]
    \item Mass is preserved: For all $t\in I^\delta$, we have 
    \begin{equation*}
    \int_{\Do}f^{\alpha,\beta,\delta,\theta}_t\dd x\dd v=1.
    \end{equation*}

    \item For $\alpha,\beta\in[0,1)$ and $\delta\in[0,T/2)$, we have
    \begin{align}
    \int_{I^\delta}\cE_{a,b}(\mu^{\alpha,\beta,\delta}_t)\dd t
    &\lesssim\int_0^T\cE_{a,b}(\mu_t)\dd t,
    \label{mo:a:b:d}\\
    \sup_{t\in I^\delta}\cE_{a,b}(\mu^{\alpha,\beta,\delta}_t)
    &\lesssim\sup_{t\in[0,T]}\cE_{a,b}(\mu_t),
    \label{mo:a:2}\\
    \sup_{t\in I^\delta}
    \|f^{\alpha,\beta,\delta}_t\|_{L^q_vL^1_x}
    &\le\sup_{t\in[0,T]}\|f_t\|_{L^q_vL^1_x},
    \label{mo:Lp}
    \end{align}
    for $a\in[0,2]$, $b\in[0,2+|\gamma|]$ and $1\le q<\infty$, whenever the right-hand sides are finite. The implicit constants are independent of the regularisation parameters.

    For $\theta\in[0,1)$, the moment bounds hold up to a constant, and we have
    \begin{equation*}
    \sup_{t\in I^\delta}
    \|f^{\alpha,\beta,\delta,\theta}_t\|_{L^q_vL^1_x}
    \le
    (1-\theta)\sup_{t\in[0,T]}\|f_t\|_{L^q_vL^1_x}
    +\theta\sup_{t\in[0,T]}\|g_t\|_{L^q_vL^1_x}.
    \end{equation*}

    \item Let $\alpha,\beta>0$. If $\theta>0$, we have
    \begin{align}
    |\log f^{\alpha,\beta,\delta,\theta}_t(x,v)|
    &\le C(\alpha,\beta,\theta,T)
    \big(1+\log(\langle v\rangle+\langle x\rangle)\big).
    \label{low-bdd:log:gamma}
    \end{align}
    If $\theta=0$, we have
    \begin{align}
    |\log f^{\alpha,\beta,\delta}_t(x,v)|
    &\le C\big(\langle v\rangle+\langle x\rangle
    +\cE_{2,0}(\mu^\delta_t)^{\frac12}\big),
    \label{low-bdd:log}\\
    |\cH(f^{\alpha,\beta,\delta}_t)|
    &\le C\cE_{2,0}(\mu^\delta_t)^{\frac12}
    \cE_{1,1}(\mu^\delta_t)<+\infty,
    \label{H:reg:badd}
    \end{align}
    where $C$ depends on $\alpha,\beta$ and the fixed velocity moment bound.
\end{enumerate}
\end{lemma}

\begin{proof}
We recall from \eqref{poly:M:bdd} that
\begin{equation}
\label{M-beta}
\big(\langle\cdot\rangle^p*M_\beta\big)(z)
\le C\langle z\rangle^p
\quad\forall\,p\in\R,
\end{equation}
where $C$ may depend on $p$ but is independent of $\beta\in(0,1)$.

\begin{enumerate}[1.]
    \item By definition.

    \item By definition and \eqref{M-beta}, we have
    \begin{align*}
    &\int_{I^\delta}\int_{\Do}
    \langle v\rangle^b\dd\mu^{\alpha,\beta,\delta}_t(x,v)\dd t\\
    ={}&
    \int_{I^\delta}\int_{\Do}
    \Big(\int_{\R^d}M_\beta(v-u)\langle v\rangle^b\dd v\Big)
    \dd\mu^\delta_t(y,u)\dd t\\
    \le{}&
    C\int_{I^\delta}\int_{-\delta}^{\delta}\eta^\delta_s
    \int_{\Do}\langle u\rangle^b\dd\mu_{t-s}(y,u)\dd s\dd t\\
    \le{}&
    C\int_0^T\int_{\Do}\langle v\rangle^b\dd\mu_t(x,v)\dd t.
    \end{align*}
    The proofs for the spatial moments and \eqref{mo:a:2} follow analogously.

    We show \eqref{mo:Lp}. Writing
    $F_t(v)=\int_xf_t(x,v)\dd x$ and
    $F^{\beta,\delta}_t(v)=\int_xf^{\beta,\delta}_t(x,v)\dd x$,
    we have
    \begin{align*}
    \|f^{\alpha,\beta,\delta}_t\|_{L^q_vL^1_x}^q
    &=
    \int_{\R^d}\big(F^{\beta,\delta}_t(v)\big)^q\dd v\\
    &=
    \int_{\R^d}
    \Big(\int_{-\delta}^{\delta}\int_{\R^d}
    F_{t-s}(u)M_\beta(v-u)\eta^\delta_s\dd u\dd s\Big)^q
    \dd v\\
    &\le
    \int_{\R^d}\int_{-\delta}^{\delta}\int_{\R^d}
    F_{t-s}^q(u)M_\beta(v-u)\eta^\delta_s
    \dd u\dd s\dd v\\
    &\le
    \sup_{t\in[0,T]}\int_{\R^d}F_t^q(v)\dd v
    =
    \sup_{t\in[0,T]}\|f_t\|_{L^q_vL^1_x}^q,
    \end{align*}
    where we use Jensen's inequality since $x\mapsto x^q$ is convex.

    It remains to show the boundedness of $\cE_{a,b}(g_t)$ and
    $\|g_t\|_{L^q_vL^1_x}$. These quantities are uniformly bounded for
    $t\in[0,T]$ by our choice of the decay exponent in \eqref{def:g}.
    The estimates for the lower bound regularisation then follow from convexity and the triangle inequality.

    \item For fixed $\alpha,\beta>0$, we have
    \begin{equation}
    \label{reg:f:upper}
    \begin{aligned}
    f^{\alpha,\beta,\delta}_t(x,v)
    &=
    \int_{\Do}M_\alpha(x-y)M_\beta(v-w)\dd\mu^\delta_t(y,w)\\
    &\le
    \|M_\alpha\|_{L^\infty}\|M_\beta\|_{L^\infty}
    \lesssim(\alpha\beta)^{-d}.
    \end{aligned}
    \end{equation}
    If $\theta>0$, we have
    \begin{equation*}
    f^{\alpha,\beta,\delta,\theta}_t(x,v)
    \ge\theta g(t,x,v)
    \ge C(\theta,T)
    \big(\langle x\rangle+\langle v\rangle\big)^{-a}.
    \end{equation*}
    Together with the upper bound, this gives \eqref{low-bdd:log:gamma}.

    In the case $\theta=0$, we show
    \begin{equation}
    \label{low-bdd:f}
    f^{\alpha,\beta,\delta}_t(x,v)
    \ge C\exp\Big(-C\big(
    \langle v\rangle+\langle x\rangle
    +\cE_{2,0}(\mu^\delta_t)^{\frac12}\big)\Big).
    \end{equation}
    We follow \cite[Lemma 30]{carrillo2024landau}, with an appropriate modification for the spatial variable. We note that
    \begin{equation*}
    \langle\frac{v-w}{\beta}\rangle
    \le\sqrt2\Big(
    \langle\frac{v}{\beta}\rangle
    +\langle\frac{w}{\beta}\rangle\Big),
    \end{equation*}
    and similarly,
    \begin{equation*}
    \langle\frac{x-y}{\alpha}\rangle
    \le\sqrt2\Big(
    \langle\frac{x}{\alpha}\rangle
    +\langle\frac{y}{\alpha}\rangle\Big).
    \end{equation*}
    We define
    $\Omega_R=\{(x,v)\in\Do:|v|\le R_1,\ |x|\le R_2\}$.
    Then we have
    \begin{equation}
    \label{reg:f:bdd:0}
    \begin{aligned}
    f^{\alpha,\beta,\delta}_t(x,v)
    \ge{}&
    C\exp\Big(-\sqrt2\big(
    \langle\frac{v}{\beta}\rangle
    +\langle\frac{x}{\alpha}\rangle
    +\langle\frac{R_1}{\beta}\rangle
    +\langle\frac{R_2}{\alpha}\rangle\big)\Big)
    \mu^\delta_t(\Omega_R).
    \end{aligned}
    \end{equation}
    On $\Dot$, the spatial coordinates are taken in a bounded fundamental domain, and the periodised kernel satisfies the same lower estimate.

    Notice that
    \begin{equation*}
    \mu^\delta_t(\Omega_R^c)
    \le
    \frac{\cE_{0,2}(\mu^\delta_t)}{R_1^2}
    +\frac{\cE_{2,0}(\mu^\delta_t)}{R_2^2}.
    \end{equation*}
    By the moment assumptions and time convolution, we can choose $R_1$ independent of $t$ and $\delta$ such that
    \begin{equation*}
    \frac{\sup_{t\in I^\delta}\cE_{0,2}(\mu^\delta_t)}{R_1^2}
    \le\frac14.
    \end{equation*}
    Taking $R_2(t)=2\cE_{2,0}(\mu^\delta_t)^{1/2}$, we have
    $\mu^\delta_t(\Omega_R)\ge1/2$.
    Substituting these choices into \eqref{reg:f:bdd:0}, we obtain
    \eqref{low-bdd:f}.

    The bound \eqref{low-bdd:log} follows from
    \eqref{reg:f:upper} and \eqref{low-bdd:f}.
    Finally, we have
    \begin{align*}
    |\cH(f^{\alpha,\beta,\delta}_t)|
    &\le
    C\cE_{2,0}(\mu^\delta_t)^{\frac12}
    \int_{\Do}f^{\alpha,\beta,\delta}_t
    \big(\langle x\rangle+\langle v\rangle\big)\dd x\dd v\\
    &\le
    C\cE_{2,0}(\mu^\delta_t)^{\frac12}
    \cE_{1,1}(\mu^\delta_t)<+\infty,
    \end{align*}
    where we use the moment bounds.
\end{enumerate}
\end{proof}

We recall the definition of the entropy dissipation:
\begin{equation}
\label{Ent-D}
\cD(f)=\frac12\int_{\G}A\kappa ff_*
\big|\Pi_{(v-v_*)^\perp}
\big(\nabla_v\log f-\nabla_{v_*}\log f_*\big)\big|^2\dd\eta.
\end{equation}
We define $\nabla_0\defeq\nabla_v-\nabla_{v_*}$. The following identity holds:
\begin{align*}
&ff_*\big|\Pi_{(v-v_*)^\perp}
\big(\nabla_v\log f-\nabla_{v_*}\log f_*\big)\big|^2\\
={}&4\big|\Pi_{(v-v_*)^\perp}\nabla_0\sqrt{ff_*}\big|^2
=\frac{\big|\Pi_{(v-v_*)^\perp}\nabla_0(ff_*)\big|^2}{ff_*}.
\end{align*}
Then \eqref{Ent-D} can be written as
\begin{equation}
\label{Ent-D-1}
\cD(f)=\frac12\int_{\G}A\kappa
\frac{\big|\Pi_{(v-v_*)^\perp}\nabla_0(ff_*)\big|^2}{ff_*}
\dd\eta.
\end{equation}

We have the following estimates on the regularised entropy dissipation and curve action. The dissipation estimate below does not include velocity or time regularisation.

\begin{lemma}
\label{lem:A-D}
For fixed $\alpha\in[0,1)$, we have
\begin{equation}
\label{pw-D}
A\kappa
\frac{\big|\Pi_{(v-v_*)^\perp}\nabla_0(f^\alpha f^\alpha_*)\big|^2}
{f^\alpha f^\alpha_*}
\le
A\kappa
\Big(
\frac{\big|\Pi_{(v-v_*)^\perp}\nabla_0(ff_*)\big|^2}{ff_*}
*_{x,x_*}M_\alpha
\Big)
\end{equation}
and
\begin{equation}
\label{bdd:D}
\int_0^T\cD(f^\alpha_t)\dd t
\le C\int_0^T\cD(f_t)\dd t.
\end{equation}
For fixed $\alpha,\beta\in[0,1)$, we have
\begin{equation}
\label{pw-A}
\frac{|U^{\alpha,\beta}|^2}
{f^{\alpha,\beta}f^{\alpha,\beta}_*}
\le
\Big(\frac{|U|^2}{ff_*}\Big)
*_{v,v_*}M_\beta*_{x,x_*}M_\alpha
\end{equation}
and
\begin{equation}
\label{bdd:AD}
\int_0^T\cA(f^{\alpha,\beta}_t,U^{\alpha,\beta}_t)\dd t
\le C\int_0^T\cA(f_t,U_t)\dd t.
\end{equation}
\end{lemma}

\begin{proof}
The spatial regularisation commutes with
$\Pi_{(v-v_*)^\perp}$ and $\nabla_0$.
By the convexity of $(a,b)\mapsto|a|^2/b$ and Jensen's inequality,
we obtain \eqref{pw-D}.
The bounds on $\kappa$ imply
\begin{equation*}
M_\alpha*_{x,x_*}\kappa\lesssim\kappa.
\end{equation*}
Hence, by \eqref{Ent-D-1} and Fubini's theorem, we have
\begin{align*}
\int_0^T\cD(f^\alpha_t)\dd t
&\le
\frac12\int_0^T\int_{\G}
A\big(M_\alpha*_{x,x_*}\kappa\big)
\frac{\big|\Pi_{(v-v_*)^\perp}\nabla_0(ff_*)\big|^2}{ff_*}
\dd\eta\dd t\\
&\le C\int_0^T\cD(f_t)\dd t.
\end{align*}

Similarly, Jensen's inequality gives \eqref{pw-A}.
Since $M_\alpha*_{x,x_*}\kappa^{-1}\lesssim\kappa^{-1}$ and
$\kappa$ is independent of the velocity variables, we have
\begin{align*}
&\int_0^T\cA(f^{\alpha,\beta}_t,U^{\alpha,\beta}_t)\dd t\\
\le{}&
\frac12\int_0^T\int_{\G}
\big(M_\alpha*_{x,x_*}\kappa^{-1}\big)
\frac{|U_t|^2}{ff_*}\dd\eta\dd t\\
\le{}&
C\int_0^T\cA(f_t,U_t)\dd t.
\end{align*}
\end{proof}

\begin{corollary}
For fixed $\alpha,\beta\in[0,1)$, $\delta\in[0,T/2)$,
$\theta\in[0,1)$ and $\gamma\in(-\gamma_d,1]$, we have
\begin{equation}
\label{int:U:gamma}
\begin{aligned}
\int_{I^\delta}\int_{\G}|U_t^{\alpha,\beta,\delta,\theta}|
\big(
\langle v\rangle^{1+\frac{|\gamma|}{2}}
+\langle v_*\rangle^{1+\frac{|\gamma|}{2}}
+\langle x\rangle+\langle x_*\rangle
\big)\dd\eta\dd t<+\infty
\end{aligned}
\end{equation}
and
\begin{equation}
\label{int:U:gamma:v-v*}
\int_{I^\delta}\int_{\G}
|U_t^{\alpha,\beta,\delta,\theta}|
|v-v_*|^{1+\frac{\gamma}{2}}\dd\eta\dd t<+\infty.
\end{equation}
\end{corollary}

\begin{proof}
In Lemma~\ref{lem:bdd:U}, we showed the corresponding weighted
integrability estimates for $U_t$.
The estimates for $U_t^{\alpha,\beta,\delta,\theta}$ follow directly
by convolution, \eqref{M-beta}, and Fubini's theorem.
For $\gamma<-2$, we also use
\begin{equation*}
|v-v_*|^{1+\frac{\gamma}{2}}*_{v,v_*}M_\beta
\lesssim|v-v_*|^{1+\frac{\gamma}{2}}.
\end{equation*}
Time convolution preserves these
integrability bounds.
\end{proof}

\subsubsection{Approximated chain rule}\label{subsec:chain-rule}

The regularised density is smooth in $t,x,v$ and satisfies
\begin{equation}
\label{app:fl}
(\d_t+v\cdot\nabla_x)f^{\alpha,\beta,\delta,\theta}_t
-\sE^{\alpha,\beta,\delta,\theta}_t
+\frac12(\widetilde\nabla\cdot\cU^{\alpha,\delta,\theta}_t)^\beta=0,
\end{equation}
where
\begin{equation*}
(\widetilde\nabla\cdot\cU^{\alpha,\delta,\theta}_t)^\beta
=
(\widetilde\nabla\cdot\cU^{\alpha,\delta,\theta}_t)*_vM_\beta
\end{equation*}
and
\begin{equation*}
\sE^{\alpha,\beta,\delta,\theta}_t(x,v)
=(1-\theta)\int_{\R^d}M_\beta(u)u\cdot
\nabla_xf^{\alpha,\delta}_t(x,v-u)\dd u.
\end{equation*}
Indeed, for any $\phi\in C^\infty_c(\Do)$, we have
\begin{align*}
&\frac{d}{dt}\int_{\Do}\phi
f^{\alpha,\beta,\delta,\theta}_t\dd x\dd v
-\int_{\Do}v\cdot\nabla_x\phi
f^{\alpha,\beta,\delta,\theta}_t\dd x\dd v\\
={}&
\int_{\Do}\phi\sE^{\alpha,\beta,\delta,\theta}_t\dd x\dd v
+\frac12\int_{\G}
\widetilde\nabla(\phi*_vM_\beta)\cdot
U^{\alpha,\delta,\theta}_t\dd\eta.
\end{align*}
Here we use that spatial convolution commutes with
$\widetilde\nabla$ and that $(\d_t+v\cdot\nabla_x)g=0$.

We first show the absolute continuity of
$t\mapsto\cH(f^{\alpha,\beta,\delta,\theta}_t)$.
We use
\begin{equation*}
|a\log a-b\log b|
\le
\big(1+\max\{|\log a|,|\log b|\}\big)|a-b|,
\qquad a,b>0.
\end{equation*}
By \eqref{low-bdd:log:gamma}, the time regularisation, and
$|\d_tg|\lesssim_Tg$, the entropy difference quotients are bounded
by an integrable multiple of
\begin{equation*}
\big(\langle x\rangle+\langle v\rangle\big)
\Big(\int_0^Tf^{\alpha,\beta}_r(x,v)\dd r+g(0,x,v)\Big).
\end{equation*}
We can therefore differentiate under the integral sign. Since mass
is preserved, we obtain
\begin{align*}
\frac{d}{dt}\cH(f^{\alpha,\beta,\delta,\theta}_t)
=
\int_{\Do}\log f^{\alpha,\beta,\delta,\theta}_t
\,\d_tf^{\alpha,\beta,\delta,\theta}_t\dd x\dd v.
\end{align*}
Using \eqref{app:fl} and integrating by parts, we have
\begin{equation}
\label{app:chain-rule:diff}
\begin{aligned}
\frac{d}{dt}\cH(f^{\alpha,\beta,\delta,\theta}_t)
={}&
\int_{\Do}\sE^{\alpha,\beta,\delta,\theta}_t
\log f^{\alpha,\beta,\delta,\theta}_t\dd x\dd v\\
&+\frac12\int_{\G}
U^{\alpha,\delta,\theta}_t\cdot
\widetilde\nabla(\log f^{\alpha,\beta,\delta,\theta}_t)^\beta
\dd\eta,
\end{aligned}
\end{equation}
since the transport contribution vanishes.

We integrate \eqref{app:chain-rule:diff} from $s$ to $t$ to obtain
\begin{equation}
\label{app:chain-rule}
\begin{aligned}
\cH(f^{\alpha,\beta,\delta,\theta}_t)
-\cH(f^{\alpha,\beta,\delta,\theta}_s)
={}&
\int_s^t\int_{\Do}
\sE^{\alpha,\beta,\delta,\theta}_r
\log f^{\alpha,\beta,\delta,\theta}_r
\dd x\dd v\dd r\\
&+\frac12\int_s^t\int_{\G}
U^{\alpha,\delta,\theta}_r\cdot
\widetilde\nabla(\log f^{\alpha,\beta,\delta,\theta}_r)^\beta
\dd\eta\dd r
\end{aligned}
\end{equation}
for all $\delta\le s\le t\le T-\delta$.

\subsubsection{Passing to the limit}\label{subsec:limit}

We pass to the limit in the order $\delta,\theta,\beta,\alpha\to0$.
We first recall the following dominated convergence theorem, which
will be used below.

\begin{theorem}[Dominated convergence theorem]
\label{DCT}
Let $f_n(x)\to f(x)$ for a.e. $x\in X$. Suppose that the sequence is
dominated by a sequence of integrable functions $\{g_n\}$ in the sense that
$|f_n(x)|\le|g_n(x)|$ for a.e. $x\in X$ and all $n\ge0$.
If $g_n\to g$ almost everywhere, $g$ is integrable, and
$\lim_{n\to\infty}\int_X|g_n|=\int_X|g|$, then
$\lim_{n\to\infty}\int_Xf_n=\int_Xf$.
\end{theorem}

Concerning the time regularisation limit, we use the following
proposition.

\begin{proposition}
\label{conv:1}
For fixed $\alpha,\beta\in(0,1)$ and $\theta\in[0,1)$, we have
\begin{equation*}
\sup_{t\in I^\delta}
\|f^{\alpha,\beta,\delta,\theta}_t
-f^{\alpha,\beta,\theta}_t\|_{L^\infty(\Do)}
\to0
\end{equation*}
as $\delta\to0$.
If $\theta>0$, then, for every bounded $\Omega\subset\Do$, we also have
\begin{equation*}
\sup_{t\in I^\delta}\sup_{(x,v)\in\Omega}
|\log f^{\alpha,\beta,\delta,\theta}_t(x,v)
-\log f^{\alpha,\beta,\theta}_t(x,v)|
\to0.
\end{equation*}
\end{proposition}

\begin{proof}
We follow \cite{carrillo2024landau}. Since
\begin{align*}
&|f^{\alpha,\beta,\delta,\theta}_t(x,v)
-f^{\alpha,\beta,\theta}_t(x,v)|\\
\le{}&
(1-\theta)\sup_{|s|\le\delta}
|f^{\alpha,\beta}_{t-s}(x,v)-f^{\alpha,\beta}_t(x,v)|,
\end{align*}
the first convergence follows from weak continuity of $\mu_t$ on
$[0,T]$ and the uniform boundedness and Lipschitz bounds of the
translated mollifiers.

For the logarithms, we have
\begin{align*}
&|\log f^{\alpha,\beta,\delta,\theta}_t(x,v)
-\log f^{\alpha,\beta,\theta}_t(x,v)|\\
\le{}&
(\theta g_t(x,v))^{-1}
|f^{\alpha,\beta,\delta,\theta}_t(x,v)
-f^{\alpha,\beta,\theta}_t(x,v)|\\
\le{}&
C(T,\theta)\big(\langle x\rangle+\langle v\rangle\big)^a
|f^{\alpha,\beta,\delta,\theta}_t(x,v)
-f^{\alpha,\beta,\theta}_t(x,v)|,
\end{align*}
which converges uniformly to zero on $I^\delta\times\Omega$.
\end{proof}

\subsubsection{Collision term}
We first show the following estimates for the regularised curves.

\begin{lemma}
\label{lem:nabla:log}
For fixed $\beta\in(0,1)$, we have
\begin{equation}
\label{bdd:nabla:log}
|\nabla_v\log f^{\alpha,\beta,\delta,\theta}|
\lesssim_T\beta^{-1},
\quad
|\widetilde\nabla(\log f^{\alpha,\beta,\delta,\theta})^\beta|
\lesssim_{\beta,T}|v-v_*|^{1+\frac{\gamma}{2}}.
\end{equation}
In particular, for $\gamma\in[-2,1]$, we have
\begin{equation*}
|\widetilde\nabla(\log f^{\alpha,\beta,\delta,\theta})^\beta|
\lesssim_{\beta,\gamma,T}
\langle v\rangle^{1+\frac{\gamma}{2}}
+\langle v_*\rangle^{1+\frac{\gamma}{2}}.
\end{equation*}
\end{lemma}

\begin{proof}
We follow \cite[Lemmas 31 and 32]{carrillo2024landau}. We note that
\begin{equation}
\label{nabla:M-beta}
\nabla_vM_\beta(v)
=-\frac{v}{\beta^2\langle v/\beta\rangle}M_\beta(v),
\quad
|\nabla_vM_\beta(v)|\le\beta^{-1}M_\beta(v).
\end{equation}
Consequently,
\begin{equation*}
|\nabla_vf^{\alpha,\beta,\delta}|
\le\beta^{-1}f^{\alpha,\beta,\delta}.
\end{equation*}
Together with $|\nabla_vg|\lesssim_Tg$, this gives the first bound
in \eqref{bdd:nabla:log}.

By definition, we have
\begin{align*}
\widetilde\nabla(\log f^{\alpha,\beta,\delta,\theta})^\beta
={}&
|v-v_*|^{1+\frac{\gamma}{2}}\Pi_{(v-v_*)^\perp}\\
&\quad\times
\big(
\nabla_v(\log f^{\alpha,\beta,\delta,\theta})^\beta
-\nabla_{v_*}(\log f^{\alpha,\beta,\delta,\theta}_*)^\beta
\big).
\end{align*}
Since convolution with $M_\beta$ preserves the bound on the
logarithmic gradient, we obtain the second bound.
\end{proof}

\medskip
\noindent\textbf{Limit $\delta,\theta\to0$:}
We show
\begin{equation}
\label{col:delta}
\begin{aligned}
&\lim_{\theta\to0}\lim_{\delta\to0}
\int_s^t\int_{\G}
\widetilde\nabla(\log f^{\alpha,\beta,\delta,\theta}_r)^\beta
\cdot U^{\alpha,\delta,\theta}_r\dd\eta\dd r\\
&=\int_s^t\int_{\G}
\widetilde\nabla(\log f^{\alpha,\beta}_r)^\beta
\cdot U^\alpha_r\dd\eta\dd r.
\end{aligned}
\end{equation}
For the $\delta$-limit, it suffices to show
\begin{align}
&(1-\theta)\Big|
\int_s^t\int_{\G}
\widetilde\nabla(\log f^{\alpha,\beta,\theta}_r)^\beta
\cdot(U^{\alpha,\delta}_r-U^\alpha_r)\dd\eta\dd r
\Big|\to0,
\label{col-delta:conv-1}\\
&(1-\theta)\Big|
\int_s^t\int_{\G}
\big(
\widetilde\nabla(\log f^{\alpha,\beta,\delta,\theta}_r)^\beta
-\widetilde\nabla(\log f^{\alpha,\beta,\theta}_r)^\beta
\big)\cdot U^{\alpha,\delta}_r\dd\eta\dd r
\Big|\to0.
\label{col-delta:conv-2}
\end{align}

By \eqref{bdd:nabla:log} and \eqref{int:U:gamma:v-v*}, we have
\begin{equation}
\label{col-deltatilde:bdd}
|\widetilde\nabla(\log f^{\alpha,\beta,\delta,\theta})^\beta|
\lesssim_{\beta,T}|v-v_*|^{1+\frac{\gamma}{2}},
\quad
|v-v_*|^{1+\frac{\gamma}{2}}|U^\alpha|
\in L^1([0,T]\times\G).
\end{equation}
Time convolution gives
\begin{equation*}
\int_s^t\int_{\G}|v-v_*|^{1+\frac{\gamma}{2}}
|U^{\alpha,\delta}_r-U^\alpha_r|\dd\eta\dd r\to0.
\end{equation*}
This proves \eqref{col-delta:conv-1} and gives uniform integrability
of the weighted regularised rates.

For \eqref{col-delta:conv-2}, we have
\begin{align*}
&\nabla_v(\log f^{\alpha,\beta,\delta,\theta})^\beta
-\nabla_v(\log f^{\alpha,\beta,\theta})^\beta\\
={}&
\int_{\R^d}
\big(\log f^{\alpha,\beta,\delta,\theta}
-\log f^{\alpha,\beta,\theta}\big)(x,u)
\nabla_vM_\beta(v-u)\dd u.
\end{align*}
Proposition~\ref{conv:1}, the uniform logarithmic growth bound
\eqref{low-bdd:log:gamma}, and the exponential decay of
$\nabla M_\beta$ imply local uniform convergence of this expression.
The same argument applies to the starred term.
Combining this convergence with \eqref{col-deltatilde:bdd} and the
weighted $L^1$ convergence of the rates, we obtain
\eqref{col-delta:conv-2}.

For the $\theta$-limit, the logarithmic gradients converge pointwise,
and we have the integrable bound
\begin{equation*}
\big|
\widetilde\nabla(\log f^{\alpha,\beta,\theta})^\beta
\cdot U^{\alpha,\theta}
\big|
\lesssim_{\beta,T}
|v-v_*|^{1+\frac{\gamma}{2}}|U^\alpha|.
\end{equation*}
Dominated convergence therefore gives \eqref{col:delta}.

\medskip
\noindent\textbf{Limit $\beta\to0$:}
We show that
\begin{equation}
\label{CT:beta}
\begin{aligned}
\lim_{\beta\to0}\int_s^t\int_{\G}
\widetilde\nabla(\log f^{\alpha,\beta})^\beta
\cdot U^\alpha\dd\eta\dd r
=
\int_s^t\int_{\G}
\widetilde\nabla\log f^\alpha\cdot U^\alpha\dd\eta\dd r.
\end{aligned}
\end{equation}
We follow \cite{carrillo2024landau}. By Young's inequality, we have
\begin{align*}
&\Big|\int_s^t\int_{\G}
\widetilde\nabla(\log f^{\alpha,\beta})^\beta
\cdot U^\alpha\dd\eta\dd r\Big|\\
&\le
\frac12\int_0^T\int_{\G}
\kappa f^\alpha f^\alpha_*
|\widetilde\nabla(\log f^{\alpha,\beta})^\beta|^2
\dd\eta\dd t
+\int_0^T\cA(f^\alpha,U^\alpha)\dd t,
\end{align*}
where the boundedness of the curve action is ensured by
Lemma~\ref{lem:A-D}.

As in \eqref{proof:cross}, the identity
$|x|^2(y\cdot\Pi_{x^\perp}y)
=\frac12\sum_{i,j=1}^d|x_iy_j-x_jy_i|^2$
implies that
\begin{align*}
&|\widetilde\nabla(\log f^{\alpha,\beta})^\beta|^2\\
={}&
\frac12\sum_{i,j=1}^d|v-v_*|^\gamma
\Big|\Big[
v-v_*,
\nabla_v\log f^{\alpha,\beta}*_vM_\beta
-\nabla_{v_*}\log f^{\alpha,\beta}_**_{v_*}M_\beta
\Big]_{ij}\Big|^2.
\end{align*}
The cross-product term has the following bounds:
\begin{align*}
&\Big|\Big[
v-v_*,
\nabla_v\log f^{\alpha,\beta}*_vM_\beta
-\nabla_{v_*}\log f^{\alpha,\beta}_**_{v_*}M_\beta
\Big]_{ij}\Big|^2\\
\le{}&
4\Big(
|[v,\nabla_v\log f^{\alpha,\beta}*_vM_\beta]_{ij}|^2
+|[v_*,\nabla_{v_*}\log f^{\alpha,\beta}_**_{v_*}M_\beta]_{ij}|^2\\
&\quad+
|[v,\nabla_{v_*}\log f^{\alpha,\beta}_**_{v_*}M_\beta]_{ij}|^2
+|[v_*,\nabla_v\log f^{\alpha,\beta}*_vM_\beta]_{ij}|^2
\Big)\\
\lesssim{}&
|[v,\nabla_v\log f^{\alpha,\beta}*_vM_\beta]_{ij}|^2
+|[v_*,\nabla_{v_*}\log f^{\alpha,\beta}_**_{v_*}M_\beta]_{ij}|^2\\
&\quad+
|v|^2|\nabla_{v_*}\log f^{\alpha,\beta}_**_{v_*}M_\beta|^2
+|v_*|^2|\nabla_v\log f^{\alpha,\beta}*_vM_\beta|^2.
\end{align*}
The corresponding limiting terms are
\begin{equation}
\label{limit:fishers}
\begin{aligned}
\frac{|[v,\nabla_vf^\alpha]_{ij}|^2}{(f^\alpha)^2}
+\frac{|[v_*,\nabla_{v_*}f^\alpha_*]_{ij}|^2}{(f^\alpha_*)^2}
+|v|^2\frac{|\nabla_{v_*}f^\alpha_*|^2}{(f^\alpha_*)^2}
+|v_*|^2\frac{|\nabla_vf^\alpha|^2}{(f^\alpha)^2},
\end{aligned}
\end{equation}
which correspond to the Fisher and cross-Fisher information in
Lemma~\ref{lem:Des}.

We only need to show the following convergences
\begin{equation}
\label{col:beta:1}
\begin{aligned}
&\lim_{\beta\to0}\int_0^T\int_{\G}
\kappa f^\alpha f^\alpha_*|v-v_*|^\gamma
|[v,\nabla_v\log f^{\alpha,\beta}*_vM_\beta]_{ij}|^2
\dd\eta\dd t\\
&=\int_0^T\int_{\G}
\kappa f^\alpha_*|v-v_*|^\gamma
\frac{|[v,\nabla_vf^\alpha]_{ij}|^2}{f^\alpha}
\dd\eta\dd t
\end{aligned}
\end{equation}
and
\begin{equation}
\label{col:beta:2}
\begin{aligned}
&\lim_{\beta\to0}\int_0^T\int_{\G}
\kappa f^\alpha f^\alpha_*|v-v_*|^\gamma|v_*|^2
|\nabla_v\log f^{\alpha,\beta}*_vM_\beta|^2
\dd\eta\dd t\\
&=\int_0^T\int_{\G}
\kappa f^\alpha_*|v-v_*|^\gamma|v_*|^2
\frac{|\nabla_vf^\alpha|^2}{f^\alpha}\dd\eta\dd t.
\end{aligned}
\end{equation}
The limits corresponding to the second and third terms in
\eqref{limit:fishers} follow analogously.

We first show \eqref{col:beta:1}.  The radiality of $M_\beta$
and integration by parts give
\begin{gather*}
[v,\nabla_v\log f^{\alpha,\beta}*_vM_\beta]_{ij}
=\big([v,\nabla_v\log f^{\alpha,\beta}]_{ij}\big)*_vM_\beta,\\
[v,\nabla_vf^{\alpha,\beta}]_{ij}
=([v,\nabla_vf^\alpha]_{ij})*_vM_\beta.
\end{gather*}
By the Cauchy--Schwarz inequality and \eqref{M-beta}, we have
\begin{equation}
\label{col:crs:bdd}
\begin{aligned}
&|\big([v,\nabla_v\log f^{\alpha,\beta}]_{ij}\big)*_vM_\beta|^2\\
\le{}&
\Big(\int_{\R^d}M_\beta(v-w)\langle w\rangle^{-\gamma}\dd w\Big)
\int_{\R^d}M_\beta(v-w)\langle w\rangle^\gamma
|[w,\nabla_w\log f^{\alpha,\beta}]_{ij}|^2\dd w\\
\lesssim{}&
\langle v\rangle^{-\gamma}
\int_{\R^d}M_\beta(v-w)\langle w\rangle^\gamma
\frac{|[w,\nabla_wf^{\alpha,\beta}]_{ij}|^2}
{(f^{\alpha,\beta})^2}\dd w.
\end{aligned}
\end{equation}
Moreover, the joint convexity of $(a,b)\mapsto|a|^2/b$
and Jensen's inequality imply that
\begin{equation}
\label{col-beta-1-1}
\begin{aligned}
\frac{|[v,\nabla_vf^{\alpha,\beta}]_{ij}|^2}
{f^{\alpha,\beta}}
&=
\frac{|([v,\nabla_vf^\alpha]_{ij})*_vM_\beta|^2}
{f^\alpha*_vM_\beta}\\
&\le
\Big(\frac{|[v,\nabla_vf^\alpha]_{ij}|^2}{f^\alpha}\Big)
*_vM_\beta.
\end{aligned}
\end{equation}
By Lemma~\ref{lem:Des} and Jensen's inequality for the spatial
regularisation, the weighted cross-Fisher information of
$f^\alpha$ is integrable in time. Thus, 
\eqref{M-beta} and dominated convergence enusre
\begin{equation}
\label{col-beta-1-2}
\begin{aligned}
&\lim_{\beta\to0}\int_0^T\int_{\Do}
\langle v\rangle^\gamma
\Big(\frac{|[v,\nabla_vf^\alpha]_{ij}|^2}{f^\alpha}
*_vM_\beta\Big)(v)\dd x\dd v\dd t\\
&=\int_0^T\int_{\Do}\langle v\rangle^\gamma
\frac{|[v,\nabla_vf^\alpha]_{ij}|^2}{f^\alpha}
\dd x\dd v\dd t<+\infty.
\end{aligned}
\end{equation}

We now compare the functions appearing in \eqref{col:beta:1}.
Using $M_\beta(v)=\beta^{-d}M(v/\beta)$, we have
\begin{equation*}
[v,\nabla_v\log f^{\alpha,\beta}*_vM_\beta]_{ij}
=\int_{\R^d}M(u)
[v-\beta u,\nabla_v\log f^{\alpha,\beta}(v-\beta u)]_{ij}
\dd u.
\end{equation*}
Then 
\eqref{col:crs:bdd} and Lemma~\ref{bdd:conv} imply
\begin{align*}
&\int_0^T\int_{\G}
\kappa f^\alpha f^\alpha_*|v-v_*|^\gamma
\big|
[v,\nabla_v\log f^{\alpha,\beta}*_vM_\beta
-\nabla_v\log f^\alpha]_{ij}
\big|^2\dd\eta\dd t\\
\lesssim{}&
\int_0^T\int_{\Do}\int_{\R^d}
f^\alpha(v)M(u)\langle v-\beta u\rangle^\gamma\\
&\quad\times
\big|
[v-\beta u,\nabla_v\log f^{\alpha,\beta}(v-\beta u)]_{ij}
-[v,\nabla_v\log f^\alpha(v)]_{ij}
\big|^2
\dd u\dd x\dd v\dd t.
\end{align*}

For each fixed $u$, we as, as $\beta\to0$,
\begin{equation*}
[v-\beta u,\nabla_v\log f^{\alpha,\beta}(v-\beta u)]_{ij}
\to[v,\nabla_v\log f^\alpha(v)]_{ij}.
\end{equation*}
By \eqref{col-beta-1-1}, we also have
\begin{align*}
&f^\alpha(v)M(u)\langle v-\beta u\rangle^\gamma
\big|
[v-\beta u,\nabla_v\log f^{\alpha,\beta}(v-\beta u)]_{ij}
-[v,\nabla_v\log f^\alpha(v)]_{ij}
\big|^2\\
\le{}&
2f^\alpha(v)M(u)\langle v-\beta u\rangle^\gamma
\Bigg(
\frac{
\Big(
\frac{|[v,\nabla_vf^\alpha]_{ij}|^2}{f^\alpha}
*_vM_\beta
\Big)(v-\beta u)}
{f^{\alpha,\beta}(v-\beta u)}
+|[v,\nabla_v\log f^\alpha(v)]_{ij}|^2
\Bigg).
\end{align*}
The majorant on the right converges almost everywhere to $4M(u)\langle v\rangle^\gamma
\frac{|[v,\nabla_vf^\alpha(v)]_{ij}|^2}{f^\alpha(v)}$, whoes  
integral equals
\begin{equation*}
4\int_0^T\int_{\Do}\langle v\rangle^\gamma
\Big(
\frac{|[v,\nabla_vf^\alpha]_{ij}|^2}{f^\alpha}
*_vM_\beta
\Big)(v)\dd x\dd v\dd t,
\end{equation*}
where we use 
 $\int_{\R^d}M(u)f^\alpha(v+\beta u)\dd u
=f^{\alpha,\beta}(v)$. 
Its convergence follows from \eqref{col-beta-1-2}.
Theorem~\ref{DCT} therefore shows that the preceding difference
estimate tends to zero. Since the limiting quadratic integral
is finite, this proves \eqref{col:beta:1}.

We show \eqref{col:beta:2} in the same way.
As in \eqref{col:crs:bdd}, we have
\begin{equation*}
|\nabla_v\log f^{\alpha,\beta}*_vM_\beta|^2
\lesssim
\langle v\rangle^{-\gamma}
M_\beta*_v
\big(\langle v\rangle^\gamma
|\nabla_v\log f^{\alpha,\beta}|^2\big).
\end{equation*}
Consequently, Lemma~\ref{bdd:conv},
Lemma~\ref{lem:reg-curve} and Fubini's theorem give
\begin{equation}
\label{use:soft}
\begin{aligned}
&\int_0^T\int_{\G}
\kappa f^\alpha f^\alpha_*|v-v_*|^\gamma|v_*|^2
|\nabla_v\log f^{\alpha,\beta}*_vM_\beta|^2
\dd\eta\dd t\\
\lesssim{}&
\sup_{t\in[0,T]}C^\pm_{\gamma,2}(f_t)
\int_0^T\int_{\Do}\langle v\rangle^\gamma
\frac{|\nabla_vf^{\alpha,\beta}|^2}{f^{\alpha,\beta}}
\dd x\dd v\dd t.
\end{aligned}
\end{equation}
Here $C^\pm_{\gamma,2}$ are given in
\eqref{cst:C+} and \eqref{cst:C-}.

Since $\nabla_vf^{\alpha,\beta}
=(\nabla_vf^\alpha)*_vM_\beta$, Jensen's inequality gives
\begin{equation*}
\frac{|\nabla_vf^{\alpha,\beta}|^2}{f^{\alpha,\beta}}
\le
\Big(\frac{|\nabla_vf^\alpha|^2}{f^\alpha}\Big)*_vM_\beta.
\end{equation*}
As in \eqref{col-beta-1-2}, we also have
\begin{align*}
&\lim_{\beta\to0}\int_0^T\int_{\Do}
\langle v\rangle^\gamma
\Big(\frac{|\nabla_vf^\alpha|^2}{f^\alpha}
*_vM_\beta\Big)(v)\dd x\dd v\dd t\\
&=\int_0^T\int_{\Do}\langle v\rangle^\gamma
\frac{|\nabla_vf^\alpha|^2}{f^\alpha}
\dd x\dd v\dd t<+\infty.
\end{align*}
Repeating the preceding dominated convergence argument with the
cross-product expressions replaced by the corresponding gradients,
and using \eqref{col:crs:bdd} and Lemma~\ref{bdd:conv}, we obtain
\begin{align*}
&\int_0^T\int_{\G}
\kappa f^\alpha f^\alpha_*|v-v_*|^\gamma|v_*|^2
|\nabla_v\log f^{\alpha,\beta}*_vM_\beta
-\nabla_v\log f^\alpha|^2\dd\eta\dd t
\to0.
\end{align*}
This proves \eqref{col:beta:2}.

\medskip
\noindent\textbf{Limit $\alpha\to0$:}
By Young's inequality and the pointwise bounds
\eqref{pw-D} and \eqref{pw-A}, we have
\begin{align*}
&|\widetilde\nabla\log f^\alpha\cdot U^\alpha|\\
\le&{}
\frac12\kappa
\Big(
ff_*|\widetilde\nabla\log f|^2
*_{x,x_*}M_\alpha
\Big)+\frac12\kappa^{-1}
\Big(
\frac{|U|^2}{ff_*}
*_{x,x_*}M_\alpha
\Big).
\end{align*}
The collision integrand and its majorants converge almost everywhere.
As in Lemma~\ref{lem:A-D}, moving the convolutions onto $\kappa$
and $\kappa^{-1}$ and using \eqref{kernel2-2}, dominated convergence
gives convergence of the integrals of the majorants.
Theorem~\ref{DCT} therefore implies
\begin{align*}
\lim_{\alpha\to0}\int_s^t\int_{\G}
\widetilde\nabla\log f^\alpha\cdot U^\alpha\dd\eta\dd r
=
\int_s^t\int_{\G}
\widetilde\nabla\log f\cdot U\dd\eta\dd r.
\end{align*}

\subsubsection{Error term}
\leavevmode\par
\medskip
\noindent\textbf{Limit $\delta\to0$:}  We show that
\begin{equation}
\label{T-delta}
\begin{aligned}
&\lim_{\delta\to0}\int_s^t\int_{\Do}
\mathsf{E}^{\alpha,\beta,\delta,\theta}_r
\log f^{\alpha,\beta,\delta,\theta}_r
\dd x\dd v\dd r\\
&=(1-\theta)\int_s^t\int_{\Do}
\log f^{\alpha,\beta,\theta}_r(x,v)
\int_{\R^d}u\cdot\nabla_xf^\alpha_r(x,v-u)
M_\beta(u)\dd u\dd x\dd v\dd r.
\end{aligned}
\end{equation}

As in \eqref{nabla:M-beta}, we have
\begin{equation*}
    |\nabla_xM_\alpha(x)|
    \lesssim\alpha^{-1}M_\alpha(x).
\end{equation*}
Combining this with the uniform logarithmic bound
\eqref{low-bdd:log:gamma}, we obtain
\begin{align*}
&\big|
\log f^{\alpha,\beta,\delta,\theta}_r(x,v)\,
u\cdot\nabla_xf^{\alpha,\delta}_r(x,v-u)M_\beta(u)
\big|\\
&\le C(\alpha,\beta,\theta,T)
\int_{\R^d}
\big(\langle v\rangle+\langle x\rangle\big)
|u|M_\beta(u)M_\alpha(y)
f^\delta_r(x-y,v-u)\dd y.
\end{align*}
Moreover, we have
\begin{align*}
&\int_s^t\int_{\G}
\big(\langle v\rangle+\langle x\rangle\big)
|u|M_\beta(u)M_\alpha(y)
|f^\delta_r-f_r|(x-y,v-u)
\dd y\dd u\dd x\dd v\dd r\\
&\lesssim
\|\langle v\rangle^2M_\beta\|_{L^1}
\|\langle x\rangle^2M_\alpha\|_{L^1}
\int_s^t\int_{\Do}
\big(\langle v\rangle+\langle x\rangle\big)
|f^\delta_r-f_r|\dd x\dd v\dd r
\to0,
\end{align*}
where we use convergence of time convolution in
$L^1([0,T];L^1_{1,1}(\Do))$ on interior time intervals.

\medskip
\noindent\textbf{Limit $\theta\to0$:} 
To show the $\theta$-limit, we only need to verify the following
integrability estimate
\begin{align*}
&\int_s^t\int_{\Do}|\log f^{\alpha,\beta}_r(x,v)|
\int_{\R^d}|u\cdot\nabla_xf^\alpha_r(x,v-u)|
M_\beta(u)\dd u\dd x\dd v\dd r\\
\le&{}C(\alpha,\beta)\int_s^t\int_{\G}
\big(\langle v\rangle+\langle x\rangle
+\cE_{2,0}(f_r)^{\frac12}\big)\\
&\quad\times|u|M_\beta(u)M_\alpha(y)
f_r(x-y,v-u)\dd y\dd u\dd x\dd v\dd r\\
\le&{}C(\alpha,\beta,T)
\|\langle v\rangle^2M_\beta\|_{L^1}
\|\langle x\rangle M_\alpha\|_{L^1}
\int_0^T\cE_{2,1}(f_t)\dd t<+\infty,
\end{align*}
where we use the pointwise-in-time bound \eqref{low-bdd:log}
for $|\log f^{\alpha,\beta}_r|$ in the first inequality.
For $\theta\in[0,1/2]$, the lower bound
$f^{\alpha,\beta,\theta}\ge\frac12f^{\alpha,\beta}$
and the uniform upper bound give
$|\log f^{\alpha,\beta,\theta}|
\le C(\alpha,\beta,T)+|\log f^{\alpha,\beta}|$.
Thus, the same integrable majorant applies, and dominated
convergence gives the $\theta$-limit.

\medskip
\noindent\textbf{Limit $\beta\to0$:}
We show that
\begin{align*}
\lim_{\beta\to0}\int_s^t\int_{\Do}
\log f^{\alpha,\beta}_r(x,v)
\int_{\R^d}u\cdot\nabla_xf^\alpha_r(x,v-u)M_\beta(u)\dd u
\dd x\dd v\dd r=0.
\end{align*}
By integration by parts in $x$, we have
\begin{align*}
&\int_s^t\int_{\Do}
\log f^{\alpha,\beta}_r(x,v)
\int_{\R^d}u\cdot\nabla_xf^\alpha_r(x,v-u)M_\beta(u)\dd u
\dd x\dd v\dd r\\
={}&
-\int_s^t\int_{\Do}\int_{\R^d}
\frac{\nabla_xf^{\alpha,\beta}_r(x,v)}
{f^{\alpha,\beta}_r(x,v)}
\cdot u f^\alpha_r(x,v-u)M_\beta(u)
\dd u\dd x\dd v\dd r.
\end{align*}
By using $|\nabla_xM_\alpha|\lesssim\alpha^{-1}M_\alpha$
and mass preservation, we have
\begin{align*}
&\Big|
\int_s^t\int_{\Do}\int_{\R^d}
\nabla_x\log f^{\alpha,\beta}_r(x,v)
\cdot u f^\alpha_r(x,v-u)M_\beta(u)
\dd u\dd x\dd v\dd r
\Big|\\
\lesssim{}&
\frac1\alpha\int_s^t\int_{\R^d}|u|M_\beta(u)
\Big(\int_{\Do}f^\alpha_r(x,v-u)\dd x\dd v\Big)
\dd u\dd r\\
={}&
\frac{t-s}{\alpha}\int_{\R^d}|u|M_\beta(u)\dd u\\
={}&
\frac{(t-s)\beta}{\alpha}
\int_{\R^d}|u|M(u)\dd u
\lesssim\frac{T\beta}{\alpha}\to0
\end{align*}
as $\beta\to0$, for fixed $\alpha>0$.

\subsubsection{Entropy term}

We first treat the limit
$\lim_{\delta\to0}\cH(f^{\alpha,\beta,\delta,\theta}_t)$
for fixed $\alpha,\beta,\theta>0$ and all $t\in(0,T)$.
The mean value theorem implies that
\begin{equation}
\label{H:delta}
\begin{aligned}
&|\cH(f^{\alpha,\beta,\delta,\theta}_t)
-\cH(f^{\alpha,\beta,\theta}_t)|\\
\le{}&
\int_{\Do}
|f^{\alpha,\beta,\delta,\theta}_t
\log f^{\alpha,\beta,\delta,\theta}_t
-f^{\alpha,\beta,\theta}_t
\log f^{\alpha,\beta,\theta}_t|\dd x\dd v\\
\le{}&
\int_{\Do}
|f^{\alpha,\beta,\delta,\theta}_t
-f^{\alpha,\beta,\theta}_t|
\Big(1+\max\big\{
|\log f^{\alpha,\beta,\delta,\theta}_t|,
|\log f^{\alpha,\beta,\theta}_t|
\big\}\Big)\dd x\dd v\\
\le{}&
C(\alpha,\beta,\theta,T)\int_{\Do}
\big(\langle v\rangle+\langle x\rangle\big)
|f^{\alpha,\beta,\delta}_t-f^{\alpha,\beta}_t|
\dd x\dd v\to0
\end{aligned}
\end{equation}
as $\delta\to0$, where the uniform logarithmic bounds are given by
\eqref{low-bdd:log:gamma}.
Indeed, the uniform second-moment bounds in
Lemma~\ref{lem:reg-curve}-2 give uniform integrability of the
weighted densities, and the uniform convergence of
$f^{\alpha,\beta,\delta}$ is given by Proposition~\ref{conv:1}.

The same bounds imply continuity of
$t\mapsto\cH(f^{\alpha,\beta,\theta}_t)$ on $[0,T]$.
Since the collision and error terms are integrable in time,
the identity obtained after $\delta\to0$ extends to
$0\le s\le t\le T$.

Notice that $\cH(f^{\alpha,\beta,\theta}_t)$ is finite for all
$t\in[0,T]$. To show the $\theta\to0$ limit, we repeat the
argument in \eqref{H:delta}, using the pointwise-in-time bound
\eqref{low-bdd:log} for $|\log f^{\alpha,\beta}_t|$,
the lower bound
$f^{\alpha,\beta,\theta}_t\ge\frac12f^{\alpha,\beta}_t$
for $\theta\in[0,1/2]$, and the finite moments of the auxiliary
density $g_t$ defined in \eqref{def:g}.

It remains to show the $\alpha,\beta\to0$ limits.
By lower semicontinuity of the entropy and
$f^{\alpha,\beta}_t\to f_t$ in $L^1_{2,2}(\Do)$, we have
\begin{align*}
    \cH(f_t)\le
    \liminf_{\alpha,\beta\to0}\cH(f^{\alpha,\beta}_t).
\end{align*}
On the other hand, the convexity of $a\mapsto a\log a$
and Jensen's inequality imply that, for any
$\alpha,\beta\in(0,1)$,
\begin{equation*}
    \cH\Big(
    \int_{\Do}f_t(x-y,v-u)M_\alpha(y)M_\beta(u)\dd y\dd u
    \Big)\le\cH(f_t).
\end{equation*}
Hence, we obtain
$\lim_{\alpha,\beta\to0}\cH(f^{\alpha,\beta}_t)=\cH(f_t)$
in the extended sense. The same argument applies to each
regularisation separately.

We take $s=0$ in the regularised identity after the
$\delta,\theta\to0$ limits and pass to the limits
$\beta\to0$ and then $\alpha\to0$.
Since $|\cH(f_0)|<+\infty$ and the collision and error limits
have been established, we obtain $|\cH(f_t)|<+\infty$ and
\eqref{int:chain-rule} for all $t\in[0,T]$.
The integrability of the right-hand side of the chain rule
then implies that $t\mapsto\cH(f_t)$ is absolutely continuous
and bounded on $[0,T]$.

\subsection{Proof of Theorem \ref{thm:main}}
\label{VC}

We first observe that $\cH(f_t)>-\infty$ for all $t\in[0,T]$.
Following \cite{JKO98}, we have
\begin{align*}
\int_{\Do}|\min\{f\log f,0\}|\dd x\dd v
\le{}&
\int_{\Do}\big(\langle x\rangle^2+\langle v\rangle^2\big)
f\dd x\dd v\\
&+\int_{\Do}
\exp\big(-1-\langle x\rangle^2-\langle v\rangle^2\big)
\dd x\dd v.
\end{align*}
Hence, we have
\begin{equation}
\label{H:ge}
\cH(f_t)\ge-C\cE_{2,2}(f_t)
\quad\forall\,t\in[0,T].
\end{equation}

We show $\cJ_T(f,U)\ge0$.
By \eqref{H:ge} and $|\cH(f_0)|<+\infty$, we have
$\cH(f_T)-\cH(f_0)>-\infty$.
If either the action or the entropy dissipation is not time
integrable, then $\cJ_T(f,U)=+\infty$.
Otherwise, Theorem~\ref{thm:chain-rule} and Young's inequality give
\begin{align*}
\cH(f_T)-\cH(f_0)
={}&
\frac12\int_0^T\int_{\G}
\widetilde\nabla\log f_t\cdot U_t\dd\eta\dd t\\
\ge{}&
-\frac12\int_0^T\cD(f_t)\dd t
-\frac12\int_0^T\cA(f_t,U_t)\dd t.
\end{align*}
This implies $\cJ_T(f,U)\ge0$.

We consider curves such that $\cJ_T(f,U)=0$.
By \eqref{H:ge}, we have
\begin{equation*}
\int_0^T\big(\cD(f_t)+\cA(f_t,U_t)\big)\dd t<+\infty.
\end{equation*}
The chain rule yields
\begin{align*}
0=\cJ_T(f,U)
=\frac14\int_0^T\int_{\G}
\Big|
\sqrt{\kappa ff_*}\,\widetilde\nabla\log f_t
+\frac{U_t}{\sqrt{\kappa ff_*}}
\Big|^2\dd\eta\dd t.
\end{align*}
Hence, we have
\begin{equation}
\label{U:f}
U=-\kappa ff_*\widetilde\nabla\log f
\end{equation}
almost everywhere on $[0,T]\times\G$.
The weak formulation \eqref{FL:H:ab} follows since
$(f,-\kappa ff_*\widetilde\nabla\log f)\in\TGRE_T$.
Together with the moment bounds, weak continuity, the entropy bound
given by the chain rule, and the finite entropy dissipation, this
implies that $f$ is a $\cH$-solution.



\section{Further discussion and outlook}
\label{sec: conclusion}

In this section, we discuss different representations of the GENERIC structure of the fuzzy Landau equation. In particular, we explain how the cross-product Landau gradient introduced in \cite{DH26} simplifies the treatment of the collision term in the chain-rule argument of Section~\ref{sec:vc}. We also discuss possible directions for future work.

\subsection{Other GENERIC building blocks}
\label{App:kernel}

An evolution equation may possess different GENERIC structures \cite{mielke2023introduction}. For instance, the equation \(\partial_t\rho=\Delta\rho\) can be written as a gradient flow in at least two different ways:
\[
\partial_t\rho=\Delta\rho
=\aM_1(\rho)\dd\aS_1(\rho)
=\aM_2(\rho)\dd\aS_2(\rho),
\]
where
\begin{gather*}
   \aM_1(\rho)\xi=-\div(\rho\nabla \xi),~\aS_1(\rho)=\int\rho\log \rho\\ \text{and}\quad \aM_2(\rho)\xi=-\div(\rho^2\nabla\xi),~\aS_2(\rho)=\int \log\rho. 
\end{gather*}

The Boltzmann equation can be formulated either as a non-quadratic GENERIC system \cite{grmela2018generic} or as a quadratic one \cite{ottinger1997generic}. For the fuzzy Landau equation, one can also distribute the interaction kernel and the projection differently between the gradient and the mobility. The resulting factorizations below represent the same dissipative operator, but their analytical properties can be different.

We recall the fuzzy Landau gradient \eqref{gradient}:
\begin{equation*}
    \widetilde\nabla f
    =\sqrt{A(v-v_*)}\Pi_{(v-v_*)^\perp}
    \big(\nabla_vf-\nabla_{v_*}f_*\big).
\end{equation*}
We consider the following families of gradients:
\begin{gather}
    \widetilde\nabla_\xi f
    =\xi(|v-v_*|)\Pi_{(v-v_*)^\perp}
    \big(\nabla_vf-\nabla_{v_*}f_*\big),
    \label{nabla:xi}\\
    \widetilde\nabla_\Xi f
    =\Xi(v-v_*)
    \big(\nabla_vf-\nabla_{v_*}f_*\big),
    \label{nabla:Xi}
\end{gather}
where
\begin{gather*}
    \xi:\R_+\to\R_+,
    \qquad
    \Xi:\R^d\to\R^{d\times d}.
\end{gather*}
We require that the nonnegative weights $\xi_1$ and $\xi_2$ satisfy
\begin{gather*}
    \xi^2(|v-v_*|)\xi_1(v-v_*)=A(v-v_*),\\
    \xi_2(v-v_*)\Xi(v-v_*)^T
    \Pi_{(v-v_*)^\perp}\Xi(v-v_*)
    =A(v-v_*)\Pi_{(v-v_*)^\perp}.
\end{gather*}
The corresponding divergence operators are defined by duality as in \eqref{tilde:IP}. For sufficiently smooth $G:\G\to\R^d$ with suitable decay, we have
\begin{gather*}
    \widetilde\nabla_\xi\cdot G
    =\nabla_v\cdot\int_{\Do}
    \xi(|v-v_*|)\Pi_{(v-v_*)^\perp}
    (G-G_*)\dd x_*\dd v_*,\\
    \widetilde\nabla_\Xi\cdot G
    =\nabla_v\cdot\int_{\Do}
    \big(\Xi(v-v_*)^TG-\Xi(v_*-v)^TG_*\big)
    \dd x_*\dd v_*,
\end{gather*}
where $G_*$ denotes $G$ with $(x,v)$ and $(x_*,v_*)$ interchanged. In particular, the second formula reduces to an expression involving $G-G_*$ when $\Xi$ is even.

We retain the spatial kernel $\kappa$ in the mobility and define
\begin{gather*}
    \aM_\xi(f)g
    =-\frac12\widetilde\nabla_\xi\cdot
    \big(\kappa\xi_1ff_*\widetilde\nabla_\xi g\big),\\
    \aM_\Xi(f)g
    =-\frac12\widetilde\nabla_\Xi\cdot
    \big(\kappa\xi_2ff_*\Pi_{(v-v_*)^\perp}
    \widetilde\nabla_\Xi g\big)
\end{gather*}
for sufficiently smooth $g$. The factorisation identities imply
\begin{equation*}
    \aM_\xi(f)=\aM_\Xi(f)=\aM(f).
\end{equation*}
Thus, these representations recover the same fuzzy Landau equation with the energy, entropy and Hamiltonian operator given in Section~\ref{subsec:main-thm}. The representation with $\Xi=\operatorname{Id}_d$ is also used in \cite{gualdani2025fuzzylandauequationglobal}.

Although these factorizations agree formally, the proof of a variational characterisation depends on the chosen representation. The main issue is passing to the limit in the collision term of the regularised chain rule.

\begin{itemize}
    \item \underline{$\widetilde\nabla_\xi$ case:}
    The dependence of $\Pi_{(v-v_*)^\perp}$ on the relative velocity prevents a direct commutation with velocity convolution. Consequently, Jensen's inequality alone does not provide a bound on the regularised entropy dissipation. In Section~\ref{sec:vc}, we overcome this difficulty through the Fisher and cross-Fisher information decomposition.

    \item \underline{$\widetilde\nabla_\Xi$ case:}
When $\Xi=\operatorname{Id}_d$, the lifted differential operator
$\nabla_v-\nabla_{v_*}$ commutes with convolution. However, the
entropy dissipation controls only its projection onto
$(v-v_*)^\perp$, and velocity convolution does not generally
preserve the tangentiality of the grazing rate.

For the kernels \eqref{A-2}--\eqref{A:Rd2:0}, the cross-product
formulation in \cite{DH26} resolves this difficulty.
In dimension three, we take
\begin{equation*}
    \widetilde\nabla_\Xi f
    =(v-v_*)\times
    \big(\nabla_vf-\nabla_{v_*}f_*\big),
    \qquad \xi_2=A_0.
\end{equation*}
Then $A_0|\widetilde\nabla_\Xi g|^2=|\widetilde\nabla g|^2$,
so we obtain the same dissipative quadratic form.
The corresponding matrix-valued formulation is discussed
in \cite{DH26}.

The key observation is that the lifted cross-product operator
commutes with Gaussian velocity convolution. By duality, we have
\begin{equation*}
    \big(\widetilde\nabla_\Xi\cdot U\big)*_vM_\beta
    =\widetilde\nabla_\Xi\cdot
    \big(U*_{v,v_*}M_\beta\big).
\end{equation*}
Moreover, $A_0\sim\langle v-v_*\rangle^\gamma$ ensures
\begin{equation*}
    A_0^{\pm1}*_{v,v_*}M_\beta\le CA_0^{\pm1},
\end{equation*}
uniformly in $\beta$.
These properties and Jensen's inequality give bounds and
convergence for the regularised action and entropy dissipation.
Young's inequality and Theorem~\ref{DCT} then yield the collision
limit without the Fisher information decomposition.

Thus, the proof combines the chain-rule argument of
Section~\ref{sub:chain-rule} with the treatment of the collision
term in \cite{DH26}. Under the moment and entropy conditions
stated there, the curve bounds \eqref{curve-2} suffice, without
additional $L^p$ assumptions. This includes the Maxwellian case.
For the other power-law kernels \eqref{A-1}, the required
convolution bounds for $A_0=|v-v_*|^\gamma$ do not generally hold.
\end{itemize}

\subsection{Other kinetic models}

Together with previous works \cite{carrillo2024landau,erbar2023gradient,EH25,DH26} on the Boltzmann and Landau equations, our results place the spatially homogeneous, fuzzy (delocalised), and classical versions of these equations into a common variational framework, with the classical formulations understood at least formally. This provides a basis for deriving one equation from another by passing to the limit in a suitable parameter. Examples include the grazing collision limit from Boltzmann equations to Landau equations and the localisation limit from fuzzy models to classical ones.

A natural direction is to extend the evolutionary $\Gamma$-convergence methods developed for gradient flows \cite{sandier2004gamma,mielke2016evolutionary} to these GENERIC variational formulations. Another interesting direction for future work is to generalise the results of this paper to relativistic and quantum Boltzmann and Landau equations, and to their multi-species counterparts.

\appendix
\section{GENERIC formulation}
\label{app:A}
In this appendix, we show how formally fuzzy Landau equation \eqref{Landau-fuz} can be cast into a GENERIC framework
\begin{equation}
\label{eq: GENERIC-Landau}
\partial_t f=\aL(f)\dd\aaE(f)+\aM(f)\dd \aS(f).    
\end{equation}

We consider the space $\aZ$ to be the space that consists of all Schwartz functions $f\in \cS(\Do)$ such that $f\cL\in\cP(\Do)$. The space is endowed with the $L^2$-inner product $\langle f,g\rangle=\int_{\Do}fg\dd x\dd v$. We formally consider the tangent and co-tangent space at $f\in\aZ$ as $T_f\aZ=\{\xi\in L^2(\Do)\mid \int_{\Do}\xi\dd x\dd v=0\}$ and $T_f^*\aZ=L^2(\Do)\}$.

In the case of the fuzzy Landau equation \eqref{Landau-fuz}, we define the energy and entropy functionals by 
\begin{align*}
&\aaE(f)=\int_{\Do} \frac{|v|^2}{2} f\dd x\dd v\quad\text{and}\quad \aS(f)=-\cH(f).
\end{align*}

We define the operators $\aL$ and $\aM$ at $f\in\aZ$ by
\begin{align*}
\aM(f)g&=-\frac{1}{2}\widetilde{\nabla}\cdot\Big(\kappa ff_*\widetilde{\nabla} g\Big)\quad\text{and}\\
\aL(f)g&=-\nabla\cdot(f \aJ\nabla g),\quad \aJ=\begin{pmatrix}
    0& \mathsf{id}_d\\
    -\mathsf{id}_d&0
\end{pmatrix}
\end{align*}

for all $g\in \aZ$, where $\nabla =(\nabla_x,\nabla_v)^T$.

\medskip

We can check directly that $\{\aL, \aM, \aaE, \aS\}$ is a GENERIC building block for the fuzzy Landau equation \eqref{Landau-fuz}. Indeed, $\aL$ is anti-symmetric and satisfies the Jacobi identity, see  \cite{DPZ13}
 \begin{align*}
&\langle \aL(f)g,g\rangle=\int_{\Do} f \nabla g \aJ\nabla g=0
\end{align*}
for all $g\in\aZ$.
 The operator $\aM$ is symmetric, positive semi-definite since, using the integration by parts formula \eqref{tilde:IP}, we have
\begin{align*}
\langle \aM(f)g,h\rangle&=-\frac{1}{2}\int_{\G}\widetilde{\nabla}\cdot\Big(\kappa ff_*\widetilde{\nabla} g\Big) h\dd\eta\\
&=\frac{1}{2}\int_{\G}\kappa f f_* \widetilde{\nabla} g\cdot \widetilde{\nabla}h\dd\eta=\langle \aM(f)h,g\rangle
\\\text{and}\quad   \langle \aM(f)g,g\rangle&=\frac{1}{2}\int_{\G} \kappa f f_* |\widetilde{\nabla} g|^2\dd\eta\ge 0
\end{align*}
for all $g,h\in\aZ$.

\medskip

We indeed recover the fuzzy Landau equation \eqref{Landau-fuz} from the GENERIC equation \eqref{eq: GENERIC-Landau} since
\begin{itemize}
    \item $\dd\aaE(f)=\frac{|v|^2}{2} $ and 
    \begin{align*}
        \aL(f)\dd\aaE(f)&=-\frac12\nabla\cdot \big(f \aJ\nabla |v|^2\big)=-v\cdot\nabla_x f,
    \end{align*}
    \item $\dd\aS(f)=-(\log f+1)$ and
    \begin{gather*}
        \aM(f)\dd\aS(f)=\frac12\widetilde\nabla\cdot\Big(\kappa ff_*\widetilde\nabla \log f\Big).
    \end{gather*}
\end{itemize}

It remains to show the non-interaction conditions. We have, by the anti-symmetry of $\aJ$ and the definition of $\widetilde\nabla$:
\begin{align*}
\aL(f)\dd\aS(f)&=-\nabla\cdot\big(f\aJ\nabla \log f\big) =0\\
\text{and}\quad\aM(f)\dd \aaE(f)&=-\frac12\widetilde\nabla\cdot\Big(\kappa ff_*\widetilde\nabla \frac{|v|^2}{2}\Big)=0.
\end{align*}
As mentioned in the introduction, by taking $x=x_*$, we formally obtain the GENERIC formulation for the classical inhomogeneous (kinetic) Landau equation. It is worth mentioning that the classical kinetic Boltzmann equation can also be formally cast into the GENERIC framework~ \cite{ottinger1997generic,grmela2018generic}.

\printbibliography

\Addresses

\end{document}